\documentclass[12pt,reqno]{amsart}

\usepackage{amssymb}
\usepackage{amsthm}
\usepackage{amsmath}
\usepackage{bm}
\usepackage{cite}
\usepackage{mathrsfs}
\usepackage{xcolor}
\usepackage{latexsym}

\usepackage{times}

\allowdisplaybreaks
\newtheorem{theorem}{Theorem}[section]

\newtheorem{definition}[theorem]{Definition}
\newtheorem{lemma}[theorem]{Lemma}
\newtheorem{proposition}[theorem]{Proposition}
\theoremstyle{remark}
\newtheorem{remark}[theorem]{Remark}

\def\d{{\rm d}}
\def \l {\langle}
\def \r {\rangle}
\def\V{{\mathbf{V}}}
\def\H{\mathbf{H}}
\def\W{\mathbf{W}}
\def\A{\mathbf{A}}

\def\f{\textbf{\textit{f}}}
\def\uu{\textbf{\textit{u}}}
\def\vv{\textbf{\textit{v}}}
\def\ww{\textbf{\textit{w}}}
\def\g{\textbf{\textit{g}}}
\def\ddt{\frac{\d}{\d t}}
\def\C{\mathcal{C}}
\def\vp{\varphi}

\def\n{\textbf{n}}
\def\E{\mathcal E}

\def\la{\lambda}

\def \au {\rm}
\def \ti {\it}
\def \jou {\rm}
\def \bk {\it}
\def \no#1#2#3 {{\bf #1} (#3), #2.}
\def \eds#1#2#3 {#1, #2, #3.}

\begin{document}

\title[Uniqueness and regularity for the NSCH system]
{Uniqueness and regularity for \\ the
Navier-Stokes-Cahn-Hilliard system}
\author[Giorgini,  Miranville \& Temam]{Andrea Giorgini$^1$, Alain Miranville$^{2,3,4}$ 
\& Roger Temam$^5$}

\address{$^1$Indiana University, Department of Mathematics \& Institute for Scientific Computing and Applied Mathematics, Bloomington, IN 47405, USA}
\email{agiorgin@iu.edu}

\address{$^2$Laboratoire de Math\'{e}matiques et Applications, Universit\'{e} de Poitiers, UMR CNRS 7348, Boulevard Marie et
Pierre Curie - T\'{e}l\'{e}port 2, F-86962, Chasseneuil Futuroscope Cedex, France}
\email{alain.miranville@math.univ-poitiers.fr}

\address{$^3$Fudan University (Fudan Fellow), Shanghai, China}

\address{$^4$Xiamen University,
School of Mathematical Sciences,
Xiamen, Fujian, China}

\address{$^5$Indiana University,
Institute for Scientific Computing and Applied Mathematics, Bloomington, IN 47405, USA}
\email{temam@indiana.edu}



\begin{abstract}
\noindent  
The motion of two contiguous incompressible and viscous 
fluids is described within the diffuse interface theory by the so-called Model H. The system consists of the Navier-Stokes equations, 
which are coupled with the Cahn-Hilliard 
equation associated to the Ginzburg-Landau free energy with physically relevant logarithmic potential. 
This model is studied in bounded smooth domains in $\mathbb{R}^d$, $d=2$ and $d=3$, and is supplemented with a no-slip condition for the velocity, homogeneous Neumann boundary conditions for the order parameter and the chemical potential, and suitable initial conditions.
We study uniqueness and regularity of weak and strong solutions.
In a two-dimensional domain, we show the uniqueness of weak solutions and the existence and uniqueness of global strong solutions originating from an initial velocity  $\uu_0 \in \V_\sigma$, namely $\uu_0\in \mathbf{H}_0^1(\Omega)$ such that $\mathrm{div}\, \uu_0=0$. In addition, we prove further regularity properties and the validity of the instantaneous separation property. In a three-dimensional domain we show the existence and uniqueness of local strong solutions with initial velocity $\uu_0 \in \V_\sigma$.
\end{abstract}

\date{\today}
\maketitle

\section{Introduction}
\label{1}
\noindent
In the diffuse interface theory, the motion of two incompressible and viscous fluids and the evolution of the interface that separates them are described by the Model H. The domain $\Omega$ of $\mathbb{R}^d$, $d=2$ or $d=3$ is filled with a mixture of two fluids with the same density; the concentrations of the fluids are $\vp_i$, $i=1,2$, where $\vp_i\in [0,1]$ and $\vp_1+\vp_2=1$. The physics of the Model H is such that the interface between the two fluids is assumed to be a narrow region with finite thickness. The concentrations are uniform (equal to $0$ or $1$) in subregions of $\Omega$, and vary steeply but continuously across the thin interface layer. This formulation allows large interface deformations and topological changes of the interfaces in the mixture.
After the seminal work \cite{HH} on critical points of single and binary fluids, a detailed derivation of the Model H was proposed
in \cite{GPV1996} and \cite{STARO} for the flow driven by capillarity forces. The model is based on the balance of mass and momentum which are combined with constitutive laws compatible with a version of the second law of thermodynamics.
Model H has been employed in several numerical studies for concrete applications. Relevant examples are interface stretching during mixing \cite{CV}, thermocapillary flows \cite{JV},
droplet formation and collision, moving contact lines and 
large-deformation flows \cite{J1,LS}. For a
review on these topics we refer the reader to \cite{AMW} and the references therein. Further generalizations of the Model H have
been discussed for fluid mixtures with different densities in \cite{AGG2012,ANTANOVSKII1995,B2001,DSS2007,LT1998}, and for contact angle problems and ternary fluids in \cite{BM2014,Kim2012} and the references therein.

Assuming that density differences are negligible, we 
consider two state variables: the volume-averaged fluid velocity $\uu=\uu(x,t)$ and the difference of the fluids concentrations (order parameter) $\vp=\vp(x,t)$, equal to $\vp_1-\vp_2$ in the notation above, where 
$x \in \Omega \subset \mathbb{R}^d$, $d=2$ or $d=3$, $\Omega$ being a bounded domain with smooth boundary $\partial \Omega$, and $t$ the time.
The evolution of the two state variables is governed by the Navier-Stokes-Cahn-Hilliard (NSCH) system, which reads in dimensionless form: 
\begin{equation}
\begin{cases}
\label{system}
\partial_t \uu+(\uu \cdot \nabla) \uu-\mathrm{div }\, (\nu(\vp)D\uu)
+\nabla \pi= \mu \nabla \vp,\\
\mathrm{div }\, \uu=0, \\
\partial_t \vp+ \uu \cdot \nabla \vp = \Delta \mu, \\
\mu= -\Delta \vp+ \Psi'(\vp),
\end{cases}
\quad \text{ in } \Omega\times (0,T),
\end{equation}
subject to the boundary and initial conditions
\begin{equation}
\label{bcic}
\begin{cases}
\uu =\mathbf{0},\quad \partial_\n \mu=\partial_\n \vp=0, \quad &\text{ on } \partial \Omega \times (0,T),\\
\uu(\cdot,0)=\uu_0,\quad \vp(\cdot,0)=\vp_0,\quad &\text{ in } \Omega.
\end{cases}
\end{equation}
Here $\n$ is the unit outward normal vector to the boundary $\partial \Omega$, $D\uu =  \frac12\big(\nabla \uu+ (\nabla \uu)^t\big)$ is the 
symmetric gradient, $\pi=\pi(x,t)$ is the pressure and
$\mu=\mu(x,t)$ is the so-called chemical potential. The potential $\Psi$ is the physically 
relevant homogeneous free energy density introduced in \cite{CH} and defined as
\begin{equation}
\label{SINGPOT}
\Psi(z)= \frac{\theta}{2} \Big((1+z)\log(1+z)+(1-z)\log(1-z)\Big)
-\frac{\theta_0}{2} z^2, \quad \forall z\in [-1,1],
\end{equation}
where $\theta$ and $\theta_0$ are related to the absolute temperature of the mixture 
and the critical temperature, respectively. 
These two constant parameters
satisfy the physical relations
$0<\theta<\theta_0$. 
This condition implies the double-well form to the potential \eqref{SINGPOT}.
The mathematical analysis of \eqref{system}-\eqref{bcic} may lead to a solution $\vp$ with arbitrary values in $\mathbb{R}$ whatever the potential $\Psi$, but we have to keep in mind that, by its very definition, $-1 \leq \vp\leq 1$ ($\pm1$ represent the pure concentrations) and we call these {\it physical} solutions. Now, assuming that $\nu_1$ and $\nu_2$ are the viscosities of the two homogeneous fluids, 
the viscosity of the mixture is modelled by the concentration dependent term $\nu=\nu(\vp)$. In the unmatched viscosity case ($\nu_1\neq \nu_2$),
a typical form for $\nu$ is the linear
combination (see, e.g., \cite{Kim2012} and Remark \ref{Remark-assump} below)
\begin{equation}
\label{VISFORM}
\nu(z)= \nu_1 \frac{1+z}{2}+ \nu_2 \frac{1-z}{2}, \quad \forall z\in [-1,1].
\end{equation}
The particular case $\nu_1=\nu_2$ is called matched viscosity case and $\nu$ is a positive constant. 

In the literature, the NSCH system has been widely studied by considering regular approximations of the logarithmic potential \eqref{SINGPOT}. Typical examples are polynomial-like functions, such as $\Psi_0(z)=\frac{\kappa}4 (z^2-\beta^2)^2$, where $\kappa>0$ is related to $\theta$ and $\theta_0$ and $\pm \beta$ are the two minima of $\Psi$. 
In the matched viscosity case, the mathematical analysis of
problem \eqref{system}-\eqref{bcic} with regular potentials is now well established, at least for classical boundary conditions. We refer  the reader to \cite{B,BBG,GG,GaG2bis,GaG3,Gibbon} (see also  \cite{BGM14,CaG,GM-jns} for the analysis of similar systems). 
In the unmatched viscosity case, the author in \cite{B} proved the global existence of weak solutions and the existence and uniqueness of strong solutions (global if $d=2$, local if $d=3$). Concerning the longtime behavior, the existence of the trajectory attractor is showed in \cite{GaG2bis}, while the convergence to equilibrium is established in \cite{ZWH} for periodic boundary conditions. However, in the case of polynomial potentials, it is worth recalling that it is not possible to guarantee the existence of {\it physical} solutions, that is solutions for which $-1\leq \vp(x,t)\leq 1$, for almost every $x\in \Omega$ and $t>0$.

On the other hand, few results are available for the original Model H with logarithmic potential \eqref{SINGPOT}. 
The NSCH system with unmatched viscosities and logarithmic potential has been only studied in \cite{A}, 
where existence of global weak ({\it physical}) solutions and existence and uniqueness of strong solutions (global if $d=2$, local if $d=3$) are shown (see \cite[Theorem 1 and 2]{A}). 
In particular, in two dimensions, assuming 
$\uu_0 \in \V_2^{1+r}(\Omega)$ 
for $r>0$, where 
$\V_2^{1+r}(\Omega)=(\V_\sigma,\W_\sigma)_{r,2}$ is an interpolation space, and $\V_\sigma$ and $\W_\sigma$ are defined below in Section \ref{2}, and assuming a natural 
higher-order condition on $\vp_0$ (cf. Theorem \ref{strong} below), the corresponding strong solution $(\uu,\vp)$ is global in time and unique.
In three dimensions, the local existence and uniqueness of strong solutions is achieved provided that the initial velocity $\uu_0$ belongs to $\V_2^{1+r}(\Omega)$ with $r>\frac12$.
The restriction on the initial velocity in $\V_2^{1+r}$ ($r>0$ if $d=2$ and $r>\frac12$ if $d=3$) is due to the uniqueness result \cite[Proposition 1]{A}, which requires 
that $\uu \in L^\infty(0,T;\mathbf{W}^{1,q}(\Omega))$, with $q>2$ if $d=2$ and $q=3$ if $d=3$, being not true for classical 
strong solutions of the Navier-Stokes equations for an initial velocity $\uu_0\in \V_\sigma$.
In addition, the author in \cite{A} shows that any weak solution is more regular on the interval $[T,\infty)$, for some $T>0$ which is not explicitly estimated. It satisfies the so-called {\it asymptotic} separation property (see\cite[Lemma 12]{A}), namely 
\begin{equation}
\label{asp}
\exists \, \delta>0, \, \exists \, T>0:\quad \| \vp(t)\|_{L^{\infty}(\Omega)}\leq 1-\delta, \quad \forall \, t \geq T. 
\end{equation}
This is a key property in order to show that any 
single trajectory converges to an equilibrium \cite[Theorem 3]{A}.
We also mention the results in \cite{ADT2014,GGM1,MT}, where the global existence of weak solutions to similar systems has been established. In \cite{ADT2014} the author considers a version of the NSCH system for non-Newtonian fluids, in \cite{GGM1} the authors study the NSCH system with boundary conditions that account for a moving contact line slip velocity, whereas in \cite{MT} the authors consider the Navier-Stokes-Cahn-Hilliard-Oono system. For the sake of completeness, we refer the interested reader to \cite{A2012,A2009,ADG2013,ADG2013-2,AF2008} for the analysis of the NSCH system with different densities.
Finally, we mention among many references \cite{BCB,BSSW,CS2016,CSW2013,CWWW2017,DXW2015,DWWW2017,
Fe,FHL,FW2012,GW,GZW2018,GLLW2017,HW2015,HHK,HW2014,J,KW2007,
KSW,KKL,LS,Sa,SY,SYY,SBW,
TLK,Liu2006,YFLS} 
for the numerical analysis, in particular stability and convergence analysis, numerical simulations and  control problems of the NSCH system.
At this stage we note that to date some important issues are still unsolved, such as the uniqueness of weak solutions of the NSCH in dimension two as well as the uniqueness of strong solutions with initial velocity in $\V_\sigma$ in both two and three dimensions. It is not even known whether such properties hold in the simpler case with matched viscosities. Besides, uniqueness of weak solutions in dimension two is an open question even for the NSCH system with regular potential and unmatched viscosities.

The aim of this work is to answer positively to the above mentioned open questions. 
Our main results for the NSCH system with unmatched viscosities are the following:
\begin{itemize}
\item[$1.$] If $d=2$, we show the uniqueness of weak ({\it physical}) solutions.
\item[$2.$] If $d=2$, we prove the global existence and uniqueness of strong solutions when $\uu_0\in \V_\sigma$.
\item[$3.$] If $d=2$, we show that any (weak or strong) solution becomes instantaneously more regular (that is on $[\tau,\infty)$ for any $\tau>0$), and it satisfies the {\it instantaneous} separation property, namely 
\begin{equation}
\label{isp}
\forall \, \tau>0, \ \exists \, \delta=\delta(\tau)>0:\quad \| \vp(t)\|_{L^{\infty}(\Omega)}\leq 1-\delta, \quad \forall \, t \geq \tau. 
\end{equation}
\item[$4.$] If $d=3$, we prove the local existence and uniqueness of strong solutions when $\uu_0 \in \V_\sigma$.
\end{itemize}
We observe that the technique here employed to prove the uniqueness of weak solutions in dimension two can be applied to show the same result for the following two cases: logarithmic potential and matched viscosities, and regular potentials and unmatched viscosities (see Remark \ref{uniq-rem1} and \ref{uniq-rem2}). It is worth mentioning that our method not only entails the uniqueness of weak solutions in dimension two, but a continuous dependence estimate on the initial data with a time-dependent exponent.

The mathematical analysis presented in this paper may be employed to investigate other diffuse interface models with logarithmic potential \eqref{SINGPOT}, also in connection with the study of optimal control problems and the analysis of numerical schemes. Among several models, we mention those systems that involve different laws for the velocity field, such as the Hele-Shaw and Brinkman approximations \cite{CG,GGW} or regularized family of the Navier-Stokes equations \cite{GM-jns} (see, also, \cite{CaG}). It would be interesting as well to analyze modified equations of the Cahn-Hilliard type \cite{BM2014,GGM,reviewCH,MT} or the Allen-Cahn equation (see, e.g., \cite{GG2}). A further important issue would be to extend the analysis to the non-isothermal version of the Model H introduced in \cite{ESR2015,ESR2016} and to the Model H with mass transfer and chemotaxis presented in \cite{LW2018}.
\smallskip

\textbf{Plan of the paper.} In Section \ref{2} we introduce the functions spaces, the main assumptions of the paper and we report a result of existence of weak solutions. In Section \ref{3} we discuss the uniqueness of weak solutions in two dimensions. Section \ref{4} is devoted to analysis of strong solutions, the instantaneous 
regularization of weak solutions and the separation property in space dimension two,.
Section \ref{6} is devoted to the study of strong solutions in space dimension three. 
We report in Appendixes \ref{Neumann-Laplace} and \ref{stokesappendix} some mathematical tools regarding the Neumann and Stokes problems.
\smallskip 
 
\section{Preliminaries}
\label{2}
\setcounter{equation}{0}

\subsection{Notation and Functions Spaces}
\noindent
Let $X$ be a (real) Banach or Hilbert space with norm denoted by $\| \cdot\|_X$. The boldface letter $\mathbf{X}$ stands for the vectorial space $X^d$ ($d$ is the spatial dimension), which consists of vector-valued functions $\uu$ with all components belonging to $X$, with norm $\| \cdot\|_{\mathbf{X}}$.
Let $\Omega$ be a bounded domain in $\mathbb{R}^d$, where $d=2$ or $d=3$, with smooth boundary $\partial \Omega$. 
We denote by $W^{k,p}(\Omega)$, $k\in \mathbb{N}$, the Sobolev space of
functions in $L^p(\Omega)$ with distributional derivatives of order less than or
equal to $k$ in $L^p(\Omega)$ and by $\| \cdot \|_{W^{k,p}(\Omega)}$ its norm.
For $k\in \mathbb{N}$, the Hilbert space $W^{k,2}(\Omega)$ is denoted 
by $H^k(\Omega)$ with norm $\|\cdot \|_{H^k(\Omega)}$.  
We denote by $H_0^1(\Omega)$ the closure of 
$\mathcal{C}_0^{\infty}(\Omega)$ in 
$H^1(\Omega)$ and by $H^{-1}(\Omega)$ its dual space.
We define $H=L^2(\Omega)$. Its inner product and norm are denoted by $( \cdot,\cdot )$ and $\| \cdot \|$,
respectively. We set $V=H^{1}(\Omega)$ with norm $\|\cdot \|_V$, and 
we denote its dual space by $V^{\prime}$ with norm $\|
\cdot \|_{V'}$. The symbol $\l \cdot, \cdot \r$ will stand
for the duality product between $V$ and $V'$.
We denote by $\overline{u}$ the average of $u$ over $\Omega$, that is $\overline{u}=|\Omega|^{-1}\l u,1\r$, for all $u\in V'$. By the generalized Poincar\'{e} inequality (see \cite[Chapter II, Section 1.4]{T}),
we recall that
$u \rightarrow (\| \nabla u\|^2+ |\overline{u}|^2)^\frac12$ is a norm on $V$ equivalent to the natural one.
We recall the following Gagliardo-Nirenberg and Agmon inequalities (see, e.g., \cite{Temam}) 
\begin{align}
\label{LADY}
&\| u\|_{L^4(\Omega)}\leq C \|u\|^{\frac12}\|u\|_V^{\frac12}, \quad &&\forall \, u \in V, \quad \text{if} \ d=2,\\
\label{LADY3}
&\| u\|_{L^3(\Omega)}\leq C \|u\|^{\frac12}\|u\|_V^{\frac12}, \quad &&\forall \, u \in V, \quad \text{if} \ d=3,\\
\label{Agmon2d}
&\| u\|_{L^\infty(\Omega)}\leq C \|u\|^{\frac12}\|u\|_{H^2(\Omega)}^{\frac12}, \quad && \forall \, u \in H^2(\Omega), \quad \text{if} \ d=2,\\
\label{DL4}%
&\|\nabla u\|_{\mathbf{L}^4(\Omega)}\leq C \| u\|_{L^\infty(\Omega)}^\frac12 \| u\|_{H^2(\Omega)}^\frac12,  \quad &&\forall \, u \in H^2(\Omega), \quad \text{if} \ d=2,3,
\end{align}
and the Brezis-Gallouet inequality (see \cite{BG1980})
\begin{equation}
\label{BGI}
\| u\|_{L^\infty(\Omega)}\leq C \| u\|_V\Big[\log \Big(\mathrm{e}+\frac{\| u\|_{H^2(\Omega)}}{\| u\|_V} \Big) \Big]^{\frac12}, \quad \forall \, u \in H^2(\Omega), \quad \text{if} \ d=2.
\end{equation}

We now introduce the Hilbert space of solenoidal vector-valued functions. We denote by $\mathcal{C}_{0,\sigma}^\infty(\Omega)$ the space of divergence free vector fields in $\mathcal{C}_{0}^\infty(\Omega)$.
We define $\H_\sigma$ and $\V_\sigma$ as the closure of $\mathcal{C}_{0,\sigma}^\infty(\Omega)$ with respect to the $\H$ and $\H_0^1(\Omega)$ norms, respectively.
We also use $( \cdot ,\cdot )$ and 
$\Vert \cdot \Vert $ for
the norm and the inner product in $\H_\sigma$. The space $\V_\sigma$ is endowed with the inner product and norm
$( \uu,\vv )_{\V_\sigma}=
( \nabla \uu,\nabla \vv )$ and  $\|\uu\|_{\V_\sigma}=\| \nabla \uu\|$, respectively.
We denote by $\V_\sigma'$ its dual space.
We recall that the Korn's inequality
entails
$$
\|\nabla\uu\|\leq \sqrt2\|D\uu\|\leq \sqrt2 \| \nabla \uu\|,
\quad \forall \, \uu \in \V_\sigma,
$$
where $D\uu =  \frac12\big(\nabla \uu+ (\nabla \uu)^t\big)$. 
In turn, the above inequality gives that $\uu \rightarrow \|D\uu\|$ is a norm on $\V_\sigma$ equivalent to the initial norm. We consider the Hilbert space
$\W_\sigma= \mathbf{H}^2(\Omega)\cap \V_\sigma$
with inner product and norm
$ ( \uu,\vv)_{\W_\sigma}=( \A\uu, \A \vv )$ and $\| \uu\|_{\W_\sigma}=\|\A \uu \|$, where $\A$ is the Stokes operator (see Appendix \ref{stokesappendix} for the definition and some properties).
We recall that there exists $C>0$ such that
\begin{equation}
\label{H2equiv}
 \| \uu\|_{H^2(\Omega)}\leq C\| \uu\|_{\W_\sigma}, \quad \forall \, \uu\in \W_\sigma.
\end{equation}
Finally, we introduce the trilinear continuous form on $\mathbf{H}_0^1(\Omega)$
$$
b(\uu,\vv,\ww)= \int_{\Omega} (\uu \cdot \nabla) \vv \cdot \ww \, \d x
= \sum_{i,j=1}^2 \int_{\Omega} u_i \frac{\partial v_j}{\partial x_i} w_j \, \d x,
\quad \uu, \vv, \ww\in \mathbf{H}_0^1(\Omega),
$$
satisfying the relation
$b(\uu,\vv,\vv)=0$, for all $\uu \in \V_\sigma$ and $\vv \in \mathbf{H}^1(\Omega)$.
\smallskip

\subsection{Main Assumptions}
We require that the viscosity $\nu \in \C^2(\mathbb{R})$ satisfies
\begin{equation}
\label{ipo-nu}
 0<2\nu_\ast\leq \nu(z)\leq \nu^\ast, \quad \forall \, z\in \mathbb{R},
\end{equation}
for some positive values $\nu_\ast, \nu^\ast$.
The singular potential $\Psi$ belongs to the class of functions 
$\mathcal{C}([-1,1])\cap
\mathcal{C}^{3}(-1,1)$ and has the form
\begin{equation}
\label{Singpot-form}
\Psi(z)= F(z) - \frac{\theta_0}{2}z^2, \quad \forall \, z \in [-1,1],
\end{equation}
with
\begin{equation}
\label{H}
\lim_{z\rightarrow -1}F'(z)=-\infty, \quad
\lim_{z\rightarrow 1}F'(z)=+\infty, \quad
 F''(z)\geq \theta>0,
\end{equation}
and 
\begin{equation}
\label{alpha}
\theta_0-\theta=\alpha>0.
\end{equation}
We define $F(z)=+ \infty$ for any $z\notin [-1,1]$.
We assume without loss of generality that $F(0)=0$.
In addition, we require that $F''$ is convex and
\begin{equation}
\label{Fsec}
F''(z)\leq C\mathrm{e}^{C|F'(z)|}, \quad \forall \, z \in (-1,1).
\end{equation}
for some positive constant $C$.
Also, we assume that there exists $\gamma \in(0,1)$ such that
$F''$ is non-decreasing in $[1-\gamma,1)$ and non-increasing in
$(-1,-1+\gamma]$.

\begin{remark}
\label{Remark-assump}
The above assumptions are satisfied and motivated by the logarithmic potential \eqref{SINGPOT}. In that case, $\Psi$ is extended by continuity at $z=\pm 1$. 
Notice also that the viscosity function \eqref{VISFORM} can be easily extended on the whole $\mathbb{R}$ in such way to comply \eqref{ipo-nu}. Moreover, other physically relevant profiles can be considered (up to a suitable extension), such as (see, e.g., \cite{GLLW2017,ESVVCDC2016})
$$
\nu(z)= \frac{\nu_1 \nu_2}{\big(\nu_1(\frac{1-z}{2})+ \nu_2(\frac{1+z}{2})\big)},\quad \text{or} \quad
\nu(z)=\nu_1 \mathrm{e}^{( \log(\frac{\nu_2}{\nu_1})(\frac{1-z}{2}))}, \quad \forall \ z \in [-1,1],
$$
where $\nu_1$ and $\nu_2$ are the constant viscosities of the  two fluids.
\end{remark}

\noindent
\textbf{General agreement.}
Throughout the paper, the symbol $C$ denotes a positive constant which may be estimated in terms of $\Omega$ and of the parameters of the system (see Main assumptions). Any further dependence will be explicitly pointed out when necessary. In particular, the notation $C=C(\kappa_1,...,\kappa_n)$ denotes a positive constant 
which explicitly depends on the quantities $\kappa_i$, $i=1,...,n$.

\subsection{Existence of Weak Solutions} 
Let us introduce the notion of weak solution. 
\begin{definition}
\label{defweak}
Let $T>0$ and $d=2,3$. Given $\uu_0\in \H_\sigma$, $\vp_0\in V\cap L^{\infty}(\Omega)$ with $\| \vp_0\|_{L^\infty(\Omega)}\leq 1$ and $|\overline{\vp}_0|<1$, a pair 
$(\uu,\vp)$ is a weak solution to \eqref{system}-\eqref{bcic} on $[0,T]$ if
\begin{align*}
&\uu \in  L^{\infty}(0,T;\H_\sigma)\cap L^{2}(0,T; \V_\sigma),\ 
\partial_t \uu \in L^{\frac{4}{d}}(0,T;\V_\sigma'),\\
&\vp \in  L^\infty(0,T; V)\cap L^2(0,T; H^2(\Omega)) \cap H^1(0,T; V'),\\
&\vp \in L^{\infty}(\Omega\times (0,T)),\quad\text{with}\quad |\varphi(x,t)|<1 
\ \text{a.e. }(x,t)\in \Omega\times (0,T),\\
\end{align*}
and satisfies
\begin{align}
\label{e1}
& \l \partial_t \uu,\vv \r + b(\uu,\uu,\vv)
+ (\nu(\vp) D\uu, D \vv)
= ( \mu\nabla \vp,\vv), \quad &&\forall \, \vv \in \V_\sigma,\\ 
\label{e2}
&\l \partial_t \vp,v\r+ ( \uu\cdot \nabla \vp,v) 
+ (\nabla \mu, \nabla v )=0, 
\quad &&\forall \, v \in V,
\end{align}
for almost every $t\in (0,T)$, where $\mu\in L^2(0,T;V)$ is given by
$
\mu=-\Delta \varphi +\Psi'(\varphi).
$
Moreover, $\partial_\n \varphi=0$ almost everywhere 
on $\partial\Omega\times(0,T)$, $\uu(\cdot,0)=\uu_0$ and 
$\varphi(\cdot,0)=\varphi_0$ in $\Omega$.
\end{definition}

\begin{remark}
\label{tensor}
Notice that equation \eqref{e1} is equivalent to
\begin{equation*}
\l \partial_t \uu,\vv \r - ( \uu\otimes \uu,\nabla \vv)
+ ( \nu(\vp) D\uu, D\vv)
= ( \nabla\vp \otimes \nabla \vp,\nabla \vv ), 
\quad \forall \, \vv \in \V_\sigma,
\end{equation*}
for almost every $t\in (0,T)$,
where $(v\otimes w)_{ij}=v_iw_j$, $i,j=1$, $2$, in light of the equalities 
\begin{equation}
(\uu\cdot\nabla)\uu=\mathrm{div }\,(\uu\otimes \uu)
\end{equation}
and
$$
\mu \nabla \vp=\nabla\Big(\frac12 |\nabla \vp|^2+\Psi(\vp) \Big) -\mathrm{div }\,(\nabla \vp \otimes \nabla \vp).
$$
\end{remark}

The following existence result of weak solutions has been proven in \cite[Theorem 1]{A} (see also \cite{MT}). 

\begin{theorem}
\label{existence}
Let $d=2,3$. Assume that $\uu_0\in \H_\sigma$, $\vp_0\in V\cap L^{\infty}(\Omega)$ with $\| \vp_0\|_{L^\infty(\Omega)}\leq 1$ and $|\overline{\vp}_0|<1$.
Then, for any $T>0$, there exists a weak solution $(\uu,\vp)$ to 
\eqref{system}-\eqref{bcic} on $[0,T]$ in the sense of Definition \ref{defweak} such that 
\begin{align}
&\uu\in \C([0,T],\H_\sigma), \ \text{if}\ d=2, \quad
\uu\in \C_w([0,T],\H_\sigma), \ \text{if}\ d=3, \label{reguweak}\\
&\vp\in \mathcal{C}([0,T],V)\cap L^4(0,T;H^2(\Omega))\cap L^2(0,T;W^{2,p}(\Omega)), \label{regvpweak}
\end{align}
where $2\leq p <\infty$ is arbitrary if $d=2$ and $p=6$ if $d=3$.
Moreover, given the energy of the system
\begin{equation}
\label{energy}
\E(\uu,\vp)=\frac12 \| \uu\|^2+
\frac12 \| \nabla \vp\|^2+ \int_{\Omega} \Psi(\vp) \,\d x,
\end{equation}
any weak solution satisfies the energy inequality
\begin{equation}
\label{EI}
\E(\uu(t),\vp(t))+ \int_\tau^t   \Big( \|\sqrt{\nu(\vp(s))}D \uu (s)\|^2
+\| \nabla \mu(s)\|^2 \Big) \, \d s \leq  \E(\uu(\tau),\varphi(\tau))
\end{equation}
for almost every $0\leq \tau<T$, including $\tau=0$, and every $t\in [\tau,T]$. If $d=2$, then 
\eqref{EI} holds with equality for every $0\leq \tau<t\leq T$.
\end{theorem}

\begin{remark} 
\label{datum}
We observe that any admissible initial condition 
in Theorem \ref{existence} is such that
$\Psi(\vp_0)\in L^1(\Omega)$, so that 
$\E(\uu_0,\vp_0)<\infty$. 
However, due to $|\overline{\vp}_0|<1$, $\vp_0$ cannot be a pure concentration, i.e. $\vp_0\equiv 1$ or $\vp_0\equiv -1$.  
\end{remark}

\begin{remark}
The regularity $\vp\in L^4(0,T;H^2(\Omega))$ is not proved in \cite{A,MT}, but it has been recently shown in \cite{GGW}. Given a weak solution $(\uu,\vp)$, it can be inferred from Theorem \ref{ell2} in Appendix \ref{Neumann-Laplace} with $f=\mu+\theta_0\vp \in L^2(0,T;V)$ and $u=\vp\in L^\infty(0,T;V)$ (cf. also \eqref{vpapprH2} below). 
\end{remark}

\section{Uniqueness of Weak Solutions in Two Dimensions} 
\label{3}
\setcounter{equation}{0}

\noindent
In this section we prove the uniqueness of weak solutions for the two-dimensional NSCH system with unmatched viscosities. The key idea is to derive a differential inequality involving norms (for the difference of two solutions) weaker than the natural ones given by the energy of the system (cf. \eqref{energy}). We take full advantage of the regularity properties of the Neumann and Stokes operators which allow us to recover coercive terms. In such a way, we are able to handle the Korteweg force (i.e. the term $\mu \nabla \vp$) in the Navier-Stokes equations and the convective terms. This technique will be also employed to show the uniqueness of strong solutions if $d=3$.

\begin{theorem}
\label{weak-uniqueness}
Let $d=2$. Given $(\uu_{0},\vp_{0})$ be such that
$\uu_{0}\in \H_\sigma$, $\vp_{0}\in V$, $\|\vp_{0}\|_{L^\infty(\Omega)}\leq 1$ and 
$|\overline{\vp}_{0}|<1$, the weak solution 
to \eqref{system}-\eqref{bcic} on $[0,T]$ 
with initial datum $(\uu_{0},\vp_{0})$ is unique.
\end{theorem}

\begin{proof}
Let $(\uu_1,\vp_1)$ and $(\uu_2,\vp_2)$ be two weak solutions to \eqref{system}-\eqref{bcic} on $[0,T]$ 
with the same initial datum $(\uu_{0},\vp_{0})$.
We define $\uu=\uu_1-\uu_2$ and $\vp=\vp_1-\vp_2$.
According to Remark \ref{tensor},  $\uu$ and $\vp$ solve
\begin{align}
 &\l \partial_t \uu,\vv \r - (\uu_1\otimes\uu,\nabla \vv)
 -( \uu\otimes\uu_2,\nabla \vv)
 +(\nu(\vp_1) D \uu,\nabla \vv) \notag\\
 &\quad +((\nu(\vp_1)-\nu(\vp_2))D\uu_2, \nabla \vv)
 =( \nabla \vp_1\otimes\nabla \vp, \nabla \vv)
+(\nabla \vp\otimes \nabla\vp_2,\nabla \vv ), 
\quad &&\forall \, \vv \in \V_\sigma, \label{e1u}\\ 
\label{e2u}
&\l \partial_t \vp,v\r+ 
(\uu_1\cdot \nabla \vp,v)
+(\uu\cdot \nabla \vp_2,v)+ (\nabla \mu, \nabla v )=0, \quad &&\forall \, v \in V, 
\end{align}
where
$\mu=-\Delta \vp+\Psi'(\vp_1)-\Psi'(\vp_2).$
Taking $v=1$ in \eqref{e2u} and observing that the integrals over $\Omega$ of $\uu_1\cdot \nabla \vp$
 and $\uu\cdot \nabla \vp_2$ vanish, 
we have $\overline{\vp}(t)= \overline{\vp}(0)=0$, for all $t \in [0,T]$.
We rewrite \eqref{e2u} as
\begin{equation}
\label{e2u2}
\l \partial_t \vp,v\r
-(\vp \uu_1,\nabla v) 
-(\vp_2 \uu, \nabla v)
+(\nabla \mu, \nabla v)=0, \quad \forall \, v \in V,
\end{equation}
and we recall the following estimates (cf. \eqref{reguweak}-\eqref{regvpweak})
\begin{equation}
\label{bound}
\|\uu_i(t)\| \leq C_0,\quad
\| \vp_i(t)\|_V\leq C_0,\quad
\| \vp_i(t)\|_{L^\infty(\Omega)}\leq 1,
\quad \forall \, t \in [0,T], \ i=1,2,
\end{equation}
where the positive constant $C_0$ depends on $\E(\uu_0,\vp_0)$.
Now, taking $v=A_0^{-1}\vp$ in \eqref{e2u2} (see Appendix \ref{Neumann-Laplace} for the definition of $A_0$)
and using \eqref{propN4}, we obtain
\begin{equation*}
\frac12 \ddt \| \vp\|_{\ast}^{2}
+(\mu, \vp)=\mathcal{I}_1+\mathcal{I}_2,
\end{equation*}
where $\| \vp\|_\ast=\| \nabla A_0^{-1}\vp\|$ and
\begin{equation}
\label{I1-2}
\mathcal{I}_1=(\vp \uu_1 ,\nabla A_0^{-1}\vp),
\quad \mathcal{I}_2=( \vp_2 \uu, \nabla A_0^{-1}\vp).
\end{equation}
By the assumptions on $\Psi$, we have
\begin{align*}
(\mu, \vp)&= \|\nabla \vp\|^2+(\Psi'(\vp_1)-\Psi'(\vp_2),\vp) \\
&\geq \|\nabla \vp\|^2 - \alpha \| \vp\|^2,
\end{align*}
where $\alpha$ is defined in \eqref{alpha}.
By definition of $A_0^{-1}$, we get 
\begin{align}
\label{interpolation}
\alpha\|\vp \|^2&= \alpha (\nabla A_0^{-1} \vp, \nabla \vp) \notag\\
&\leq \frac12 \|\nabla \vp\|^2 +
 \frac{\alpha^2}{2} \| \vp\|_{\ast}^2,
\end{align}
and we end up with 
\begin{equation}
\label{uniq1}
\frac{1}{2} \ddt \| \vp\|_{\ast}^{2}+ \frac12 \| \nabla \vp\|^2\leq \frac{\alpha^2}{2} \|\vp\|_{\ast}^2+\mathcal{I}_1+\mathcal{I}_2. 
\end{equation}
Taking $\vv= \A^{-1} \uu$ in \eqref{e1u} (see Appendix \ref{stokesappendix} for the definition of $\A$), we find
\begin{equation}
\label{uniq2}
\frac12 \ddt \| \uu\|_{\sharp}^2+ (\nu(\vp_1)D \uu, \nabla \A^{-1} \uu)= \mathcal{I}_3+\mathcal{I}_4+\mathcal{I}_5,
\end{equation}
where $\|\uu \|_{\sharp}=\| \nabla \A^{-1}\uu\|$ and
\begin{align*}
\mathcal{I}_3&=-((\nu(\vp_1)-\nu(\vp_2))D\uu_2, \nabla \A^{-1} \uu),\\
\mathcal{I}_4&=(\uu_1\otimes \uu,\nabla \A^{-1} \uu)+ ( \uu\otimes \uu_2,\nabla \A^{-1} \uu),\\
\mathcal{I}_5&=( \nabla \vp_1\otimes \nabla \vp,\nabla \A^{-1} \uu)
+( \nabla \vp\otimes \nabla \vp_2,\nabla \A^{-1} \uu).
\end{align*}
Recalling that $\mathrm{div}\, (^t(\nabla \vv))=\nabla( \mathrm{div}\,\vv)$ and $\A^{-1} \uu \in L^2(0,T;D(\A))$, and integrating by parts, we obtain
\begin{align}
\label{cont1}
(\nu(\vp_1)D \uu, \nabla \A^{-1} \uu)&
= (\nabla \uu, \nu(\vp_1)D \A^{-1} \uu)\notag \\
&= -(\uu, \mathrm{div }\, (\nu(\vp_1)D \A^{-1} \uu) )\notag \\
&=-(\uu, \nu'(\vp_1) D \A^{-1}\uu \nabla \vp_1)
- \frac12 (\uu, \nu(\vp_1) \Delta \A^{-1}\uu).
\end{align}
By the properties of the Stokes operator (cf. Appendix \ref{stokesappendix}), there exists 
$p\in L^2(0,T;V)$ such that $-\Delta \A^{-1}\uu
+\nabla p= \uu$ almost everywhere in $\Omega \times (0,T)$. 
By \eqref{esthigh} and \eqref{pL2}, we have
\begin{equation}
\label{pressure}
\| p \| \leq C \|\nabla \A^{-1} \uu\|^\frac12 \| \uu\|^\frac12, \quad 
\| p \|_{V}\leq C \| \uu\|.
\end{equation}
Therefore, we are led to
\begin{align}
\label{cont2}
-\frac12(\uu, \nu(\vp_1)\Delta  \A^{-1} \uu)
&=\frac12 (\nu(\vp_1)\uu, \uu)- \frac12 (\nu(\vp_1)\uu, \nabla p)\notag\\
&\geq \nu_\ast \| \uu\|^2+\frac12( \nu'(\vp_1)\nabla \vp_1 \cdot \uu, p).
\end{align}
Here we have used $\mathrm{div} \, \uu=0$.
We now set
$$
\mathcal{H}(t)= \frac12\| \uu(t)\|_{\sharp}^2+ \frac12 \| \vp(t)\|_{\ast}^2,
$$
and 
$$
\mathcal{I}_6=( \uu, \nu'(\vp_1) D \A^{-1} \uu \nabla \vp_1),\quad
\mathcal{I}_7=-\frac12( \nu'(\vp_1)\nabla \vp_1 
\cdot \uu, p).
$$
Summing \eqref{uniq1} and \eqref{uniq2}, in light of \eqref{cont1} and \eqref{cont2} we arrive at
\begin{align}
\label{inequniq}
\ddt \mathcal{H} +\nu_\ast \| \uu\|^2+ \frac12 \| \nabla \vp\|^2\leq \alpha^2 \mathcal{H} + \sum_{k=1}^7 \mathcal{I}_k,
\end{align}
where $\mathcal{I}_1$ and $\mathcal{I}_2$ are defined in \eqref{I1-2}.
We proceed by estimating all the remainder terms on the right-hand side of \eqref{inequniq}. Hereafter the positive constant $C_i$, $i\in \mathbb{N}$, depends on $\nu_\ast$, $\nu'$, $\Omega$, $C_0$ and the constants that appear in the mentioned embedding results and interpolation inequalities.
By the embedding $V \hookrightarrow L^6(\Omega)$, the Poincar\'{e} inequality and the uniform bound \eqref{bound}, we have
\begin{align*}
\mathcal{I}_1&\leq \| \vp\|_{L^6(\Omega)} \|\uu_1 \|_{L^3(\Omega)}
\| \vp\|_{\ast}\\
&\leq \frac{1}{8} \| \nabla \vp\|^2+ 
C_1 \|\uu_1 \|_{\mathbf{L}^3(\Omega)}^2 \| \vp\|_{\ast}^2,
\end{align*}
and
\begin{align*}
\mathcal{I}_2&\leq \| \vp_2\|_{L^\infty(\Omega)} \| \uu\|\| \vp\|_{\ast}\\
& \leq \frac{\nu_\ast}{8} \| \uu\|^2+ C_2 \| \vp \|_{\ast}^2,
\end{align*}
By \eqref{LADY}, \eqref{H2equiv} and \eqref{bound}, we get
\begin{align*}
\mathcal{I}_4&\leq \Big( \| \uu_1\|_{\mathbf{L}^4(\Omega)} 
+ \| \uu_2\|_{\mathbf{L}^4(\Omega)}\Big) \| \uu\| 
\| \nabla \A^{-1} \uu\|_{\mathbf{L}^4(\Omega)} \\
&\leq C \Big( \| \uu_1\|^\frac12 \| \uu_1\|_{\V_\sigma}^\frac12+ \| \uu_2\|^\frac12 \| \uu_2\|_{\V_\sigma}^\frac12 \Big)   \| \uu\|_{\sharp}^{\frac12} \| \uu\|^{\frac32}\\
&\leq \frac{\nu_\ast}{8} \| \uu\|^2+ C_3 \Big( \| \uu_1\|_{\V_\sigma}^2+\| \uu_2\|_{\V_\sigma}^2\Big) \| \uu\|_{\sharp}^2,
\end{align*}
and
\begin{align*}
\mathcal{I}_5&\leq \Big( \| \nabla \vp_1\|_{\mathbf{L}^\infty(\Omega)}+
 \| \nabla \vp_2\|_{\mathbf{L}^\infty(\Omega)} \Big) \| \nabla \vp\| \|\nabla \A^{-1} \uu\| \\
&\leq \frac{1}{8} \| \nabla \vp\|^2+ 
C_4 \Big( \| \nabla \vp_1\|_{\mathbf{L}^\infty(\Omega)}^2
+\| \nabla \vp_2\|_{\mathbf{L}^\infty(\Omega)}^2 \Big) \| \uu\|_{\sharp}^2,
\end{align*}
Being $\nu'$ globally bounded, by using \eqref{DL4} and the estimates for the pressure \eqref{pressure}, we find
\begin{align*} 
\mathcal{I}_6&\leq C\|\uu\| \| D \A^{-1} \uu\| \|\nabla \vp_1\|_{\mathbf{L}^\infty(\Omega)}\\
&\leq \frac{\nu_\ast}{8} \| \uu\|^2 + C_5 \| \nabla\vp_1 \|_{\mathbf{L}^\infty(\Omega)}^2 \|\uu\|_{\sharp}^2,
\end{align*}
and
\begin{align*}
\mathcal{I}_7&\leq C\| \nabla \vp_1\|_{\mathbf{L}^4(\Omega)} 
\| \uu\| \| p\|_{L^4(\Omega)}\\
&\leq C \|\vp_1\|_{L^\infty(\Omega)}^\frac12 \|\vp_1\|_{H^2(\Omega)}^\frac12 \| \uu\| \| p\|^\frac12 \| p\|_V^\frac12\\
&\leq C \|\vp_1\|_{H^2(\Omega)}^\frac12
\| \nabla \A^{-1}\uu\|^\frac14 \| \uu\|^\frac74\\
&\leq \frac{\nu_\ast}{8} \| \uu\|^2+ C_6 \| \vp_1\|_{H^2(\Omega)}^4 \| \uu\|_{\sharp}^2.
\end{align*}
Finally, regarding $\mathcal{I}_3$, by using \eqref{BGI}, we obtain
\begin{align*}
\mathcal{I}_3&= \Big( \int_0^1 \nu'\big(s \vp_1+(1-s)\vp_2\big)\, \d s \, \vp D \uu_2, \nabla \A^{-1}\uu\Big)\\
&\leq C \| D\uu_2\|  \| \vp\|_{L^\infty(\Omega)}
\| \nabla \A^{-1} \uu\|\\
&\leq C_7 \| \uu_2\|_{\V_\sigma} \| \nabla \vp\| \Big[\log\Big( \mathrm{e}+ \frac{\| \vp\|_{H^2(\Omega)}}{\| \nabla \vp\|}\Big) \Big]^{\frac12}\| \uu\|_{\sharp}.
\end{align*}
Note that, when $\vp\equiv 0$, the logarithmic term on the 
right hand side is assumed to be zero. 
Collecting the above estimates, we find the differential inequality
\begin{equation}
\label{inequniq2}
\ddt \mathcal{H} +\frac{\nu_\ast}{2}\| \uu\|^2+\frac14 \| \nabla \vp\|^2
\leq  \mathcal{Y}_1 \mathcal{H}+  C_7 \| \uu_2\|_{\V_\sigma}  \Big[ \mathcal{H} \| \nabla \vp\|^2 \log\Big( \mathrm{e}+ \frac{\| \vp\|_{H^2(\Omega)}}{\| \nabla \vp\|}\Big) \Big]^{\frac12},
\end{equation}
where 
\begin{align*}
\mathcal{Y}_1(t)&=C_8 \Big(1+ \| \uu_1(t)\|_{\mathbf{L}^3(\Omega)}^2
+   \| \uu_1(t)\|_{\V_\sigma}^2 +\| \uu_2(t)\|_{\V_\sigma}^2\\
& \quad +\| \nabla \vp_1(t)\|_{\mathbf{L}^\infty(\Omega)}^2
+ \| \nabla \vp_2(t)\|_{\mathbf{L}^\infty(\Omega)}^2+ \|\vp_1 (t)\|_{H^2(\Omega)}^4\Big).
\end{align*}
Thanks to Theorem \ref{existence} and the Sobolev embedding 
$W^{2,3}(\Omega)\hookrightarrow W^{1,\infty}(\Omega)$, valid in space dimension two, we deduce that $\mathcal{Y}_1$ belongs to $L^1(0,T)$. 
In addition, recalling from \eqref{bound} that 
$\| \nabla \vp\|\leq C_0$, we have
\begin{align*}
\log\Big( \mathrm{e}+ \frac{\| \vp\|_{H^2(\Omega)}}{\| \nabla\vp\|}\Big)
&\leq \log\Big( \frac{ C_8\big( \| \nabla \vp\|+\| \vp\|_{H^2(\Omega)}\big) }{\|\nabla \vp\|^2}\Big).
\end{align*} 
Therefore, denoting
$$
\mathcal{G}(t)= \frac14 \| \nabla \vp(t)\|^2, \quad \mathcal{Y}_2(t)= C_7 \| \uu_2(t)\|_{\V_\sigma}, \quad \mathcal{S}(t)= \frac{C_8}{4} \Big(\|\nabla \vp(t)\|+\| \vp(t)\|_{H^2(\Omega)}\Big),
$$
we rewrite the differential inequality \eqref{inequniq2} as follows
\begin{equation}
\label{finaldiffeq}
\ddt \mathcal{H} +\mathcal{G}
\leq  \mathcal{Y}_1 \mathcal{H}+ 
\mathcal{Y}_2  \Big[ \mathcal{H} \mathcal{G} \log\Big( \frac{\mathcal{S}}
{\mathcal{G}}\Big) \Big]^{\frac12}.
\end{equation}
Note that $\frac{\mathcal{S}}{\mathcal{G}}\geq 1$ for the choice of $C_8$. Since $\mathcal{Y}_2 \in L^2(0,T)$, $\mathcal{S}\in L^1(0,T)$ and $\mathcal{H}(0)=0$,
we can apply \cite[Lemma 2.2]{LiTiti2016} to conclude that
$\mathcal{H}(t)=0$ for all $t\in [0,T]$, which implies the uniqueness of weak solutions.
\end{proof}

\begin{remark}
\label{uniq-rem1}
An immediate consequence of the argument performed in the proof of Theorem \ref{weak-uniqueness} is the uniqueness of weak solutions to the NSCH system in dimension two with matched viscosities (i.e. $\nu(s)=1$). 
In that particular case, let us consider $(\uu_1,\vp_1)$ and $(\uu_2,\vp_2)$ are two weak solutions to \eqref{system}-\eqref{bcic} on $[0,T]$ with initial data $(\uu_{01},\vp_{01})$ and $(\uu_{02},\vp_{02})$, respectively, where $(\uu_{0i},\vp_{0i})$, $i=1,2$, comply the assumptions of Theorem \ref{existence} and $\overline{\vp}_{01}=\overline{\vp}_{02}$. 
Then, following line by line the above proof and observing that $\mathcal{I}_3=0$, we end up with the differential inequality
$$
\ddt \mathcal{H} \leq \mathcal{Y}_1 \mathcal{H},
$$
where $\mathcal{H}$ and $\mathcal{Y}_1$ are defined above. Hence, we can infer from the Gronwall lemma the following continuous dependence estimate 
\begin{align*}
\| \uu_1(t)&-\uu_2(t)\|_{\V_\sigma'}+\| \vp_1(t)-\vp_2(t)\|_{V_0'}\leq C \| \uu_{01}-\uu_{02}\|_{\V_\sigma'}+ C\| \vp_{01}-\vp_{02}\|_{V_0'}, \quad \forall \, t\in [0,T].
\end{align*}
where $C$ is a positive constant depending on $T$ and $\E(\uu_{0i},\vp_{0i})$, $i=1,2$, but is independent of the specific form of the initial data.
\end{remark}

\begin{remark}
\label{uniq-rem2}
The proof of Theorem \ref{weak-uniqueness} also allows us to deduce the uniqueness of weak solutions to problem \eqref{system}-\eqref{bcic} with unmatched viscosities and regular potential (cf. $\Psi_0$ in Introduction). The only changes in the proof arise from the different regularity of weak solutions. Indeed, the global bound in $L^\infty$ is not known in this case, but any weak solution satisfies $\vp\in L^2(0,T;H^3(\Omega))$ (see \cite{B,GaG3}). Thus, the two terms which need a different control are $\mathcal{I}_2$ and $\mathcal{I}_7$.
Nonetheless, they can be simply estimated in the following way
\begin{align*}
\mathcal{I}_2\leq \frac{\nu_\ast}{8}\|\uu \|^2+C_2 \| \vp_2\|_{L^\infty(\Omega)}^2 \|\vp \|_\ast^2,
\end{align*}
and, by using \eqref{LADY} and $\vp\in L^\infty(0,T;V)$,
\begin{align*}
\mathcal{I}_7&\leq C \|\nabla \vp_1 \|_{\mathbf{L}^4(\Omega)} \| \uu\| \| p\|_{L^4(\Omega)}\\
&\leq C \| \nabla \vp_1\|^{\frac12} \| \vp_1\|_{H^2}^{\frac12} \| \uu\| \| p\|^{\frac12} \| p\|_V^{\frac12}\\
&\leq C\| \vp_1\|_{H^2(\Omega)}^{\frac12}\| \nabla \A^{-1}\uu\|^{\frac14} \| \uu\|^{\frac74}\\
&\leq \frac{\nu_\ast}{4} \| \uu\|^2+ C_6 \| \vp_1\|_{H^2(\Omega)}^4 \| \uu\|_{\sharp}^2.
\end{align*}
Since by interpolation $\vp_i\in L^2(0,T;L^\infty(\Omega)) \cap L^4(0,T;H^2(\Omega))$, $i=1,2$, it is easily seen that we end up with a differential equation having the same form of \eqref{finaldiffeq}.
\end{remark}

\begin{remark}
In the three dimensional case, the above proof does not allow us to deduce even a weak-strong uniqueness property, which is classical with the Navier-Stokes equations, that is the weak solution is unique when a strong solution exists.
In this case, this is due to the form of $\mathcal{I}_4$ involving both $\uu_1$ and $\uu_2$. Hence, we only expect a (conditional) uniqueness result provided that both solutions $\uu_1$ and $\uu_2$ are more regular than Definition \ref{defweak} (at least $\uu_1$, $\uu_2$ satisfy the classical condition in \cite[Remark 3.81]{Temam}).
\end{remark}

We conclude this section with a continuous dependence estimate in dual space norms with a time-dependent double exponential growth. 
  
\begin{proposition}
\label{prop-uniq}
Let $d=2$. Consider two initial data $(\uu_{01},\vp_{01})$ and $(\uu_{02},\vp_{02})$ such that $\uu_{0i}\in \H_\sigma$, $\vp_{0i}\in V$, $\| \vp_{0i}\|_{L^\infty(\Omega)}\leq 1$ and $\overline{\vp}_{01}=\overline{\vp}_{02}\in (-1,1)$. 
The weak solutions $(\uu_1,\vp_1)$, $(\uu_2,\vp_2)$ on $[0,T]$ to \eqref{system}-\eqref{bcic} with initial data $(\uu_{01},\vp_{01})$ and $(\uu_{02},\vp_{02})$, respectively, satisfy the continuous dependence estimate
\begin{equation}
\label{uniqlow}
 \mathcal{H}(t)
\leq C \Big( \frac{\mathcal{H}(0)}{C}\Big)^{ \mathrm{e}^{-\int_0^t \mathcal{Y}_3(s)\, \d s}}
, \quad 
\forall \, t\in [0,T],
\end{equation}
where 
\begin{align*}
\mathcal{H}(t)&= \frac{1}{2}\| \uu_1(t)-\uu_2(t)\|_{\sharp}^2+ \frac12 \| \vp_1(t)-\vp_2(t)\|_{\ast}^2,\\
\mathcal{Y}_3(t)&= C \Big(1+  \| \uu_1(t)\|_{\V_\sigma}^2 +\| \uu_2(t)\|_{\V_\sigma}^2 +\| \nabla \vp_1(t)\|_{\mathbf{L}^\infty(\Omega)}^2
+ \| \nabla \vp_2(t)\|_{\mathbf{L}^\infty(\Omega)}^2+ \|\vp_1 (t)\|_{H^2(\Omega)}^4 \Big), 
\end{align*}
and $C$ is a positive constant depending on the norms of the initial data.
\end{proposition}
\begin{proof}
The argument is based on an estimate of $\mathcal{I}_3$ different than that in the proof of Theorem \ref{weak-uniqueness}. Thanks to the product estimate \eqref{logprod2} in Appendix \ref{Log-est}, using the properties of $A_0$ and $\A$ (see Appendixes \ref{Neumann-Laplace} and \ref{stokesappendix}) and \eqref{bound}, we have
\begin{align*}
\mathcal{I}_3&\leq C \| D\uu_2\| \| \vp \nabla \A^{-1}\uu\|\\
&\leq C \| D\uu_2\| \| \nabla \vp\| \Big( \| \nabla \A^{-1}\uu\|+ \| \vp\|_{-1} \Big) \Big[ \log \Big( \mathrm{C} \frac{\| \uu\|+ \| \vp\|_V}{\| \nabla \A^{-1}\uu\|+ \| \vp\|_{-1}} \Big) \Big]^{\frac12}\\
&\leq \frac14 \| \nabla \vp\|^2+ 
C_9 \| D\uu_2\|^2 \mathcal{H} \log \Big( \frac{C_{10}}{\mathcal{H}} \Big). 
\end{align*}
Noting that $\mathcal{H}\leq C_{11}$ by \eqref{bound}, we observe that $C_{10}$ can be chosen sufficiently large such that $\log (\frac{C_{10}}{\mathcal{H}})\geq 1$. Exploiting the above estimate in the proof of Theorem \ref{weak-uniqueness}, we eventually deduce the refined differential inequality for the difference of two solutions (cf. \eqref{finaldiffeq})
\begin{equation}
\label{ref-diffineq}
\ddt \mathcal{H}
\leq  \mathcal{Y}_3 \mathcal{H}\log \Big( \frac{C_{10}}{\mathcal{H}} \Big).
\end{equation}
After integration of \eqref{ref-diffineq}, we obtain the following estimate
\begin{equation}
\label{ref-diffineq2}
 \mathcal{H}(t)
\leq C_{10} \Big( \frac{\mathcal{H}(0)}{C_{10}}\Big)^{ \mathrm{e}^{-\int_0^t \mathcal{Y}_3(s)\, \d s}}
, \quad 
\forall \, t\in [0,T],
\end{equation}
where $\mathcal{Y}_3\in L^1(0,T)$, for any $T>0$, due to the regularity in Theorem \ref{existence}. Noticing that $\mu(s)=s\log (\frac{C}{s})$ is an Osgood modulus of continuity, the above \eqref{ref-diffineq2}-\eqref{ref-diffineq} also imply the uniqueness of weak solutions.
\end{proof}

\begin{remark}
\label{Had}
We note that the estimate for the difference of two solutions \eqref{uniqlow} is not sufficient to guarantee the continuity of solutions with respect to the data in the norm of the energy space $\H_\sigma\times V$. A similar remark holds for the constant viscosity case (cf. Remark \ref{uniq-rem1}). Nevertheless, the continuous dependence in the energy space will be recovered by using the propagation of regularity and an interpolation technique in Section \ref{4}.
\end{remark}

\section{Global Strong Solutions and Regularity in Two Dimensions}
\label{4}
\setcounter{equation}{0}

\noindent
In this section we prove the global well-posedness of strong solutions for the NSCH system with unmatched viscosities in dimension two. Later on, some consequences will be inferred regarding the regularity and continuous dependence from the initial data. 

\begin{theorem}
\label{strong}
Let $d=2$, $\uu_0 \in \V_\sigma$ and $\vp_0 \in H^2(\Omega)$ be such that 
$\| \vp_0\|_{L^\infty(\Omega)}\leq 1$, $|\overline{\vp}_0|<1$, $\mu_0=-\Delta \vp_0+\Psi'(\vp_0) \in V$ and 
$\partial_\n \vp_0=0$ on $\partial \Omega$.
Then, for any $T>0$, there exists a unique strong solution to \eqref{system}-\eqref{bcic} on $[0,T]$ such that 
\begin{align*}
&\uu \in L^\infty(0,T;\V_\sigma)\cap L^2(0,T;\W_\sigma)\cap H^1(0,T;\H_\sigma), \quad \pi \in L^2(0,T;V),\\
&\vp \in L^\infty(0,T;W^{2,p}(\Omega))\cap H^1(0,T;V),\\
&\mu \in L^{\infty}(0,T;V)\cap L^2(0,T;H^3(\Omega))\cap H^1(0,T;V'),
\end{align*}
where $2\leq p<\infty$. The strong solution satisfies \eqref{system} 
almost everywhere in $\Omega \times (0,T)$ and  $
\partial_\n \mu=0$ almost everywhere on 
$\partial\Omega\times(0,T)$. In addition, given two strong solutions $(\uu_1,\vp_1)$, $(\uu_2,\vp_2)$ on $[0,T]$ with initial data $(\uu_{01},\vp_{01})$ and $(\uu_{02},\vp_{02})$, respectively, we have the continuous dependence estimate
\begin{equation}
\label{uniqhigh}
\| \uu_1(t)-\uu_2(t)\|+ \| \vp_1(t)-\vp_2(t)\|\leq C 
\| \uu_{01}-\uu_{02}\|+ C\| \vp_{01}-\vp_{02}\|, \quad 
\forall \, t\in [0,T],
\end{equation}
where $C$ is a positive constant depending on $T$ and on the norms of the initial data.
\end{theorem}

Let us briefly explain some technical points of the proof of Theorem \ref{strong}. The argument relies on a-priori higher-order energy estimates in Sobolev spaces, combined with a suitable approximation of the logarithmic potential and the initial datum. More precisely, we approximate the logarithmic potential $\Psi$ by means of a family of regular potentials $\Psi_\varepsilon$ defined on the whole real line. 
Next, we need to perform a suitable cut-off procedure of the initial condition, since we cannot control immediately the norm of 
$\nabla \mu_\varepsilon(0)= \nabla(-\Delta \vp_0+\Psi_\varepsilon'(\vp_0))$ with $\nabla \mu(0)= \nabla(-\Delta \vp_0+\Psi'(\vp_0))$. To overcome this difficulty, we construct a preliminary approximation of the initial datum by exploiting the regularity theory of the Neumann problem with a logarithmic nonlinearity given in Appendix \ref{Neumann-Laplace}. Our argument differs from the one used in \cite{A}, which is based on fractional time regularity and maximal regularity of a Stokes operator with variable viscosity. 

\begin{proof}[Proof of Theorem \ref{strong}]
We divide the proof into several steps. 
\smallskip

\noindent
\textbf{1. Approximation of the logarithmic potential.}
We introduce a family of regular potentials
$ \Psi_\varepsilon $ that approximate the singular potential $\Psi$.
For any $\varepsilon \in (0,1)$, we set
\begin{equation}
\Psi_\varepsilon(z)=F_\varepsilon(z)-\frac{\theta_0}{2}z^2,\quad \forall\, z\in \mathbb{R},
\label{defF}
\end{equation}
where
\begin{equation}
F_\varepsilon(z)=
\begin{cases}
 \displaystyle{\sum_{j=0}^2 \frac{1}{j!}}
 F^{(j)}(1-\varepsilon) \left[z-(1-\varepsilon)\right]^j,
 \quad &\forall\,z\geq 1-\varepsilon,\\
 F(z),\quad & \forall\, z\in[-1+\varepsilon, 1-\varepsilon],\\
 \displaystyle{\sum_{j=0}^2
 \frac{1}{j!}} F^{(j)}(-1+\varepsilon)\left[ z-(-1+\varepsilon)\right]^j,
 \quad &\forall\, z\leq -1+\varepsilon.
 \end{cases}
 \label{defF1s}
\end{equation}
By virtue of the assumptions on $\Psi$ stated in Section \ref{2},
we infer that there exists $\varepsilon^\ast \in (0,\gamma)$ (where $\gamma$ is defined in Section \ref{2}) such that, for any  
$\varepsilon\in (0,\varepsilon^\ast]$, the approximating
function $\Psi_\varepsilon$  satisfies
$\Psi_\varepsilon \in \mathcal{C}^{2}(\mathbb{R})$ and
\begin{equation}
\label{Psieps-prop}
-\widetilde{\alpha}
\leq \Psi_\varepsilon(z), 
\quad -\alpha\leq \Psi''_\varepsilon(z)\leq L,\quad \forall \, z\in \mathbb{R},
\end{equation}
where $\widetilde{\alpha}$ is a positive constant
independent of $\varepsilon$, $\alpha$ is given by \eqref{alpha}, and $L$ is a positive 
constant that may depend on $\varepsilon$. Moreover, we have that
$\Psi_{\varepsilon}(z)\leq \Psi(z)$, for all $z\in [-1,1]$, and 
$|\Psi'_{\varepsilon}(z)| \leq |\Psi'(z)|$, for all $z\in (-1,1)$ (see, e.g., \cite{FG}). 
\smallskip

\noindent
\textbf{2. Approximation of the initial datum.}
We perform a cutoff procedure on the initial condition.
To do so,
we introduce the globally Lipschitz function 
$h_k: \mathbb{R}\rightarrow \mathbb{R}$, $k\in \mathbb{N}$, such that
\begin{equation}
\label{trunc}
h_k(z)=
\begin{cases}
-k,&z<-k, \\
z,& z\in [-k,k],\\
k,&z>k.
\end{cases}
\end{equation}
We define $\tilde{\mu}_{0,k}=h_k \circ \tilde{\mu}_0$, where 
$\tilde{\mu}_0=-\Delta \vp_0+ F'(\vp_0)$ (or equivalently 
$\tilde{\mu}_0=\mu_0+\theta_0 \vp$). Since $\tilde{\mu}_0\in V$, 
the classical result on compositions in Sobolev spaces \cite{Stampacchia} yields
$\tilde{\mu}_{0,k} \in V$, for any $k>0$, and 
$
\nabla \tilde{\mu}_{0,k}=\nabla \tilde{\mu}_0 \cdot\chi_{[-k,k]} (\tilde{\mu}_0),
$
which, in turn, gives 
\begin{equation}
\label{mukH1}
\|\tilde{\mu}_{0,k}\|_V\leq \| \tilde{\mu}_0\|_V.
\end{equation}
For $k\in \mathbb{N}$, we consider the Neumann problem
\begin{equation}
\label{ellapp}
\begin{cases}
-\Delta \vp_{0,k}+F'(\vp_{0,k})=\tilde{\mu}_{0,k},\quad &\text{ in }\Omega,\\
\partial_\n \vp_{0,k}=0, \quad &\text{ on }\partial \Omega.
\end{cases}
\end{equation}
Thanks to Lemma \ref{existenceNP}, there exists a 
unique solution to \eqref{ellapp} such that 
$\vp_{0,k}\in H^2(\Omega)$, 
$F'(\vp_{0,k}) \in H$, 
which satisfies \eqref{ellapp} almost everywhere in $\Omega$ and 
$\partial_\n \vp_{0,k}=0$ almost everywhere on $\partial \Omega$.
In addition, by \eqref{estbase} and \eqref{mukH1}, we have
\begin{equation}
\label{psikH2}
\| \vp_{0,k}\|_{H^2(\Omega)}\leq C(1+\| \tilde{\mu}_{0}\|).
\end{equation}
Since $\tilde{\mu}_{0,k}\rightarrow \tilde{\mu_0}$ in $H$,  Lemma \ref{existenceNP} also entails that 
$\vp_{0,k} \rightarrow \vp_0$ in $V$. As a consequence, there exist $\tilde{m}\in (0,1)$, which is independent of $k$, and $\overline{k}$ sufficiently large such that 
\begin{equation}
\label{psiH1}
\| \vp_{0,k}\|_V \leq 1+\| \vp_0\|_V,
\quad
|\overline{\vp}_{0,k}|\leq \tilde{m}<1,
\quad \forall \, k> \overline{k}.
\end{equation}
In addition, by Theorem \ref{ell2} with 
$f=\tilde{\mu}_{0k}$, we obtain 
$$
\| F'(\vp_{0,k})\|_{L^{\infty}(\Omega)}
\leq \| \tilde{\mu}_{0,k}\|_{L^{\infty}(\Omega)}\leq k.
$$
As a byproduct, there 
exists $\delta=\delta(k)>0$ such that
\begin{equation}
\label{psikinf}
\| \vp_{0,k}\|_{L^\infty(\Omega)}\leq 1-\delta.
\end{equation}
At this point, since $F'(\vp_{0,k})\in V$, it is easily seen that $\vp_{0,k}\in H^3(\Omega)$. 
Finally, for any $\varepsilon\in (0,\overline{\varepsilon})$, where 
$\overline{\varepsilon}=\min \lbrace \frac12 \delta(k), 
\varepsilon^\ast\rbrace$, since $F(z)=F_\varepsilon(z)$ for all $z\in [-1+\varepsilon,1-\varepsilon]$,
we infer from \eqref{psikinf} that
$
-\Delta \vp_{0,k}+
F'_\varepsilon(\vp_{0,k})=\tilde{\mu}_{0,k}, 
$
which entails 
\begin{equation}
\label{muk2H1}
\|-\Delta \vp_{0,k}+
F'_\varepsilon(\vp_{0,k})\|_V\leq \| \tilde{\mu}_0\|_V.
\end{equation}
\smallskip

\noindent
\textbf{3. Approximating problems.} 
Let us introduce the Galerkin scheme. We consider the family of 
eigenfunctions $\lbrace w_j\rbrace_{j\geq 1}$ of the homogeneous Neumann operator $A_1=-\Delta+I$ (see Appendix \ref{Neumann-Laplace}) and the family of eigenfunctions 
$\lbrace \ww_j\rbrace_{j\geq 1}$ of the Stokes operator $\A$ (see Appendix \ref{stokesappendix}). In 
particular, we recall that $w_1=1$ while any $w_i$, $i>1$, 
is non-constant with $\overline{w}_i=0$.  
For any integer $n\geq 1$, we define the finite-dimensional subspaces of $V$ and $\V_\sigma$, respectively, by 
$ V_n=\text{span}\lbrace w_1,...,w_n\rbrace $ and 
$\V_n= \text{span}\lbrace \ww_1,...,\ww_n\rbrace$.
We denote by $\Pi_n$ and $P_n$ the orthogonal projections on $V_n$ and $\V_n$ 
with respect to the inner product in $H$ and in $\H_\sigma$, respectively. 
We consider the approximating sequences
\begin{equation}
\label{sepvar}
\uu_{k,\varepsilon}^n(x,t)=\sum_{i=1}^{n}g_{i}(t)\ww_i(x),  \quad 
\vp_{k,\varepsilon}^n(x,t)= \sum_{i=1}^{n}k_{i}(t)w_i(x), 
\quad 
\mu_{k,\varepsilon}^n(x,t)= \sum_{i=1}^{n}l_{i}(t)w_i(x),
\end{equation}
solutions of the following approximating system
\begin{align}
\label{e1app}
& \l \partial_t \uu_{k,\varepsilon}^n, \vv\r 
+ b(\uu_{k,\varepsilon}^n,\uu_{k,\varepsilon}^n,\vv)
+ (\nu(\vp_{k,\varepsilon}^n) 
D \uu_{k,\varepsilon}^n,D \vv)
= ( \mu_{k,\varepsilon}^n
\nabla \vp_{k,\varepsilon}^n,\vv),
 \quad &\forall \, \vv \in \V_n&, \\ 
\label{e2app}
&\l \partial_t \vp_{k,\varepsilon}^n,v\r
+ ( \uu_{k,\varepsilon}^n\cdot 
\nabla \vp_{k,\varepsilon}^n,v) 
+ (\nabla \mu_{k,\varepsilon}^n, \nabla v )=0, \quad &\forall \, v \in V_n&, 
\end{align}
where
\begin{equation}
\label{muapp}
\mu_{k,\varepsilon}^n= \Pi_n \big(-\Delta \vp_{k,\varepsilon}^n 
+\Psi_\varepsilon'(\vp_{k,\varepsilon}^n)\big).
\end{equation} 
The initial conditions are defined as 
\begin{equation}
\label{Approxinitcond}
\uu_{k,\varepsilon}^n(0)=P_n \uu_0
\quad \text{and}\quad \vp_{k,\varepsilon}^n(0)=\Pi_n\vp_{0,k}.
\end{equation}
Let us notice that $\vp_{0,k}\in H^3(\Omega)$ with $\partial_\n \vp_{0,k}=0$ 
on $\partial \Omega$. Since $D(A_1^{\frac32})= \lbrace u \in H^3(\Omega): \partial_\n u=0 \text{ on } \partial \Omega \rbrace$, we have that $\vp_{k,\varepsilon}^n(0) 
\rightarrow \vp_{0,k}$ in 
$H^3(\Omega)$ as $n\rightarrow \infty$. In turn, this gives
$\vp_{k,\varepsilon}^n(0) \rightarrow \vp_{0,k}$ in 
$L^\infty(\Omega)$. Hence, there exist $\overline{m}$, with 
$\tilde{m}<\overline{m}<1$ (independent of $n$),
and $\overline{n}$ such that
\begin{equation}
\label{Linfdatvare}
|\overline{\vp}_{k,\varepsilon}^n(0)|\leq \overline{m}, \quad 
\| \vp_{k,\varepsilon}^n(0)\|_{L^\infty(\Omega)}\leq 1-\frac12\delta(k), \quad \forall \, n>\overline{n}.
\end{equation}
On account of Steps $1$ and $2$, for any $k >\overline{k}$, we fix 
$\varepsilon\in (0, \overline{\varepsilon})$ with 
$\overline{\varepsilon}$ depending on $k$, 
and $n>\overline{n}$ with $\overline{n}$ depending on $k$.
The existence of a sequence of functions 
$\uu_{k,\varepsilon}^n,\vp_{k,\varepsilon}^n$ and $\mu_{k,\varepsilon}^n$ of the form 
\eqref{sepvar} which satisfy \eqref{e1app}-\eqref{Approxinitcond} for any $t \in [0,T]$ 
can be proved in a standard way (see, e.g., \cite{Temam}). 
In particular, the system \eqref{e1app}-\eqref{Approxinitcond} is equivalent to a Cauchy problem for a nonlinear system of ordinary differential equations in the unknowns $g_i$, $k_i$ and $l_i$, $i=1,...,n$. Thanks to the Cauchy-Lipschitz theorem, for any $n>\overline{n}$, there exists a unique maximal solution to this system defined on some interval $[0,t_n]$. Moreover, by the energy estimates we shall prove in the next step (cf. \eqref{dissest1app}), it is clear that $t_n=T$.
\smallskip

\noindent
\textbf{4. Energy estimates.}
\noindent
Let us recall the above choices of the parameters, namely for any $k >\overline{k}$, we fix 
$\varepsilon\in (0, \overline{\varepsilon})$ 
and $n>\overline{n}$.
We now show uniform energy estimates with respect to the approximating parameters $k$, $\varepsilon$ and $n$.
In particular, $c_i$, $i\in \mathbb{N}$, denotes a positive constant, which depends on the parameters of the system, the constants arising from embedding and interpolation results and the energy $\E(\uu_0,\vp_0)$, but is independent of the approximation parameters $k$, $\varepsilon$ and $n$.

First, by taking $v=1$ in \eqref{e2app}, we have 
$|\overline{\vp}_{k,\varepsilon}^n(t)|=|\overline{\vp}_{k,\varepsilon}^n(0)|\leq \overline{m}$, for all $t \geq [0,T]$.
We introduce the approximated energy
$$
\E_\varepsilon(\vv,\psi)=\frac12\| \vv\|^2+\frac12 \| \nabla \psi\|^
2+ \int_\Omega \Psi_\varepsilon(\psi)\, \d x.
$$
In light of \eqref{psiH1}, \eqref{Linfdatvare} and $\Psi_{\varepsilon}(z)\leq \Psi(z)$, for all $z\in [-1,1]$, we deduce that
\begin{align}
\label{E0}
\E_\varepsilon(\uu_{k,\varepsilon}^n(0),\vp_{k,\varepsilon}^n(0))
&=\frac12\| P_n \uu_0\|^2+ 
\frac12 \| \nabla \Pi_n \vp_{0,k}\|^2+ \int_{\Omega} \Psi_\varepsilon(\vp_{k,\varepsilon}^n(0)) \, \d x\notag\\
&\leq \frac12 \| \uu_0\|^2+ \frac12 \| \vp_{0}\|_V^2+C.
\end{align}
Here we have used that $\Psi$ is bounded on $[-1,1]$.
Taking $\vv=\uu_{k,\varepsilon}^n$ in \eqref{e1app}, 
$v =\mu_{k,\varepsilon}^n$ in \eqref{e2app}, 
multiplying \eqref{muapp} by $\partial_t 
\vp_{k,\varepsilon}^n$, and summing up the 
resulting equations, we find 
\begin{equation}
\label{eneqappr}
\ddt \E_\varepsilon(\uu_{k,\varepsilon}^n,
\vp_{k,\varepsilon}^n)
+ \|\sqrt{\nu(\vp_{k,\varepsilon}^n)} 
D\uu_{k,\varepsilon}^n \|^2
+\| \nabla \mu_{k,\varepsilon}^n\|^2  = 0,
\end{equation}
for almost every $t\in (0,T)$. 
Owing to the Korn inequality and \eqref{E0}, after an 
integration in time, we have
\begin{equation}
\label{dissest1app}
\E_\varepsilon(\uu_{k,\varepsilon}^n(t),\vp_{k,\varepsilon}^n(t)) 
+  \int_0^t \Big( \nu_\ast \| \nabla \uu_{k,\varepsilon}^n(s)\|^2
+\| \nabla \mu_{k,\varepsilon}^n(s)\|^2 \Big)
\, \d s \leq c_0, \quad \forall \, t \in [0,T].
\end{equation}
In particular, by using \eqref{Psieps-prop}, we have 
\begin{equation}
\label{uniformapp}
 \| \uu_{k,\varepsilon}^n (t) \|+ \| \vp_{k,\varepsilon}^n(t)\|_V \leq c_1, \quad \forall \, t \in [0,T].
\end{equation}
In order to find an estimate on $\| \mu_{k,\varepsilon}^n\|_V$, we recall the inequality (see, e.g., \cite[Proposition A.1]{MZ})  
$$
\| \Psi_\varepsilon(\vp_{k,\varepsilon}^n)\|_{L^1(\Omega)}
\leq c_2 \Big( 1+ ( \Psi_\varepsilon'(\vp_{k,\varepsilon}^n), 
\vp_{k,\varepsilon}^n-\overline{\vp}_{k,\varepsilon}^n) \Big),
$$
where $c_2$ depends on $\overline{m}$.
Testing \eqref{muapp} by $\vp_{k,\varepsilon}^n- \overline{\vp}_{k,\varepsilon}^n$, we 
obtain
\begin{align*}
\| \nabla \vp_{k,\varepsilon}^n\|^2+ ( \Psi_\varepsilon'(\vp_{k,\varepsilon}^n), 
\vp_{k,\varepsilon}^n-\overline{\vp}_{k,\varepsilon}^n)
=(\mu_{k,\varepsilon}^n-\overline{\mu}_{k,\varepsilon}^n, \vp_{k,\varepsilon}^n-\overline{\vp}_{k,\varepsilon}^n).
\end{align*}
Thus, by the Poincar\'{e} inequality and \eqref{uniformapp}, we have
$$
( \Psi_\varepsilon'(\vp_{k,\varepsilon}^n), 
\vp_{k,\varepsilon}^n-\overline{\vp}_{k,\varepsilon}^n)
\leq c_3 \| \nabla \mu_{k,\varepsilon}^n\|.
$$ 
Accordingly, since 
$|\overline{\mu}_{k,\varepsilon}|= |\overline{\Psi'(\vp_{k,\varepsilon}^n)}|$, we learn that
\begin{equation}
\label{dissest3app}
\| \mu_{k,\varepsilon}^n\|_V
\leq c_4(1+\| \nabla \mu_{k,\varepsilon}^n\|).
\end{equation}
Next, testing \eqref{muapp} by 
$-\Delta \vp_{k,\varepsilon}^n$ and integrating by parts, 
we get
$$
\| \Delta \vp_{k,\varepsilon}^n\|^2
+ (\Psi_\varepsilon''(\vp_{k,\varepsilon}^n)
\nabla \vp_{k,\varepsilon}^n,\nabla \vp_{k,\varepsilon}^n)
= (\nabla \mu_{k,\varepsilon}^n, \nabla \vp_{k,\varepsilon}^n).
$$
By using \eqref{Psieps-prop} and \eqref{uniformapp}, we deduce that
\begin{equation}
\label{vpapprH2}
\| \vp_{k,\varepsilon}^n\|_{H^2(\Omega)}^2 \leq c_5 ( 1+ \|\nabla \mu_{k,\varepsilon}^n \|).
\end{equation}
On the other hand, by comparison in \eqref{e1app} and in \eqref{e2app} and by exploiting \eqref{LADY}, \eqref{uniformapp} and \eqref{dissest3app}, we infer that
\begin{equation}
\label{utappr}
\| \partial_t \uu_{k,\varepsilon}^n\|_{\V_\sigma'}\leq 
c_6 (1+ \| \nabla \uu_{k,\varepsilon}^n\| + \| \nabla \mu_{k,\varepsilon}^n\| ),
\end{equation}
and
\begin{equation}
\label{vptappr}
\| \partial_t \vp_{k,\varepsilon}^n\|_{\ast}
\leq c_7 (\| \nabla \uu_{k,\varepsilon}^n\|
+ \| \nabla \mu_{k,\varepsilon}^n\|).
\end{equation}
In light of the above estimates \eqref{dissest1app}-\eqref{vptappr}, we have
\begin{align*}
& \uu_{k,\varepsilon}^n \ \text{is uniformly bounded in }L^{\infty}(0,T;\H_\sigma)\cap L^2(0,T;\V_\sigma)\cap
H^1(0,T;\V_\sigma'),\\
& \varphi_{k,\varepsilon}^n \ 
\text{is uniformly bounded in } L^\infty(0,T;V)\cap L^4(0,T;H^2(\Omega))
\cap H^1(0,T;V'),\\
& \mu_{k,\varepsilon}^n \ \text{is uniformly bounded in }L^2(0,T;V),
\end{align*}
with respect to the parameters $k$, $\varepsilon$ and $n$.
\smallskip

\noindent
\textbf{5. Higher-order energy estimates.}
We are now in position to prove uniform higher-order Sobolev estimates.
We will denote by $c'_i$, $i\in \mathbb{N}$, a positive constant, which depends on the parameters of the system, the constants arising from embedding and interpolation results, and $\E(\uu_0,\vp_0)$, but are independent of the approximation parameters $k$, $\varepsilon$ and $n$ and of the norms $\| \uu_0\|_{\V_\sigma}$ and $\| \mu_0\|_V$.

Taking $v = \partial_t \mu_{k,\varepsilon}^n$ in \eqref{e2app}, we obtain
$$
\frac{1}{2} \ddt \|\nabla \mu_{k,\varepsilon}^n\|^2+ 
( \partial_t \mu_{k,\varepsilon}^n,\partial_t\vp_{k,\varepsilon}^n)
+(\partial_t \mu_{k,\varepsilon}^n, \uu_{k,\varepsilon}^n\cdot \nabla \vp_{k,\varepsilon}^n)=0.
$$
Since $\overline{\partial_t \vp_{k,\varepsilon}^n}(t)=0$ for all $t\in [0,T]$, we have
$$
\alpha \| \partial_t \vp_{k,\varepsilon}^n\|^2
\leq \frac12 \|\nabla \partial_t \vp_{k,\varepsilon}^n\|^2
+ \frac{\alpha^2}{2} \| \partial_t \vp_{k, \varepsilon}^n\|_{\ast}^2.
$$
Then, we infer from the assumptions on $\Psi_\varepsilon$ that
\begin{align*}
( \partial_t \mu_{k,\varepsilon}^n, 
\partial_t \vp_{k,\varepsilon}^n) &= 
\|\nabla \partial_t \vp_{k,\varepsilon}^n\|^2
+ (\Psi_\varepsilon''(\vp_{k,\varepsilon}^n)
\partial_t \vp_{k,\varepsilon}^n,  \partial_t \vp_{k,\varepsilon}^n)\\
&\geq  \|\nabla \partial_t \vp_{k,\varepsilon}^n\|^2
- \alpha \| \partial_t \vp_{k,\varepsilon}^n\|^2\\
&\geq \frac12 \|\nabla \partial_t \vp_{k,\varepsilon}^n\|^2
- \frac{\alpha^2}{2} \| \partial_t \vp_{k,\varepsilon}^n\|_{\ast}^2.
\end{align*}
Besides, we observe that
\begin{align*}
( \partial_t \mu_{k,\varepsilon}^n, 
\uu_{k,\varepsilon}^n \cdot 
\nabla \vp_{k,\varepsilon}^n )
&=\ddt \Big[ (\uu_{k,\varepsilon}^n
\cdot \nabla \vp_{k,\varepsilon}^n,
\mu_{k,\varepsilon}^n ) \Big] \\
&\quad -(\partial_t \uu_{k,\varepsilon}^n \cdot 
\nabla \vp_{k,\varepsilon}^n,
\mu_{k,\varepsilon}^n )
-(\uu_{k,\varepsilon}^n\cdot\nabla 
\partial_t \vp_{k,\varepsilon}^n,
\mu_{k,\varepsilon}^n).
\end{align*}
By \eqref{dissest3app}, we get
\begin{align*}
(\mu_{k,\varepsilon}^n, 
\uu_{k,\varepsilon}^n\cdot \nabla \partial_t \vp_{k,\varepsilon}^n)
&\leq \| \uu_{k,\varepsilon}^n\|_{\mathbf{L}^3(\Omega)}
\|\nabla\partial_t \vp_{k,\varepsilon}^n\|
 \| \mu_{k,\varepsilon}^n\|_{L^6(\Omega)}\\
&\leq \frac14\| \nabla \partial_t\vp_{k,\varepsilon}^n\|^2
+ c'_1 \| \uu_{k,\varepsilon}^n\|_{\mathbf{L}^3(\Omega)}^2
(1+ \|\nabla \mu_{k,\varepsilon}^n\|^2).
\end{align*}
Accordingly, by using \eqref{vptappr}, we arrive at
\begin{align}
\label{diffineq1app}
\ddt \Big[ &\frac12 \| \nabla \mu_{k,\varepsilon}^n\|^2
+ ( \uu_{k,\varepsilon}^n\cdot 
\nabla \vp_{k,\varepsilon}^n,
\mu_{k,\varepsilon}^n)
\Big]
+ \frac14 \| \nabla 
\partial_t \vp_{k,\varepsilon}^n\|^2 \notag\\
&\leq (\partial_t \uu_{k,\varepsilon}^n \cdot \nabla \vp_{k,\varepsilon}^n, \mu_{k,\varepsilon}^n)+ 
c'_2 (1+\| \uu_{k,\varepsilon}^n\|_{\mathbf{L}^3(\Omega)}^2)
(1+ \| \nabla \uu_{k,\varepsilon}^n\|^2+\|\nabla \mu_{k,\varepsilon}^n\|^2).
\end{align}
Taking $\vv =\partial_t \uu_{k,\varepsilon}^n$ in \eqref{e1app}, we have
$$
\| \partial_t \uu_{k,\varepsilon}^n\|^2+
b(\uu_{k,\varepsilon}^n, \uu_{k,\varepsilon}^n,\partial_t \uu_{k,\varepsilon}^n)
-(\mathrm{div }\, (\nu(\vp_{k,\varepsilon}^n)D\uu_{k,\varepsilon}^n), \partial_t \uu_{k,\varepsilon}^n)
= (\mu_{k,\varepsilon}^n \nabla \vp_{k,\varepsilon}^n, \partial_t \uu_{k,\varepsilon}^n).
$$
By \eqref{LADY}, \eqref{H2equiv}, \eqref{ipo-nu} and \eqref{uniformapp}, we deduce that
\begin{align*}
b(\uu_{k,\varepsilon}^n, \uu_{k,\varepsilon}^n,
\partial_t \uu_{k,\varepsilon}^n)&\leq 
\| \uu_{k,\varepsilon}^n\|_{\mathbf{L}^4(\Omega)}
 \| \nabla  \uu_{k,\varepsilon}^n\|_{\mathbf{L}^4(\Omega)} 
\|\partial_t \uu_{k,\varepsilon}^n \|\\
&\leq \sqrt{c_1} C \| \nabla  \uu_{k,\varepsilon}^n \| 
\| \A \uu_{k,\varepsilon}^n\|^{\frac12} 
\|\partial_t \uu_{k,\varepsilon}^n \|\\
&\leq \frac16  \|\partial_t \uu_{k,\varepsilon}^n \|^2+ 
c'_3 \Big( \| \A \uu_{k,\varepsilon}^n\|^2+ \| \nabla  \uu_{k,\varepsilon}^n \|^4 \Big),
\end{align*}
and
\begin{align*}
(\mathrm{div }\, (\nu(\vp_{k,\varepsilon}^n)
D\uu_{k,\varepsilon}^n), 
\partial_t \uu_{k,\varepsilon}^n)
&=\frac12( \nu(\vp_{k,\varepsilon}^n) \Delta \uu_{k,\varepsilon}^n,  \partial_t \uu_{k,\varepsilon}^n)+
(\nu'(\vp_{k,\varepsilon}^n)
  D \uu_{k,\varepsilon}^n \nabla \vp_{k,\varepsilon}^n, 
\partial_t \uu_{k,\varepsilon}^n)\\
&\leq  C \| \A \uu_{k,\varepsilon}^n\| \|\partial_t \uu_{k,\varepsilon}^n \|
+ C \| \nabla \vp_{k,\varepsilon}^n\|_{\mathbf{L}^4(\Omega)}
\| D\uu_{k,\varepsilon}^n\|_{\mathbf{L}^4(\Omega)} 
\|\partial_t \uu_{k,\varepsilon}^n \|\\
&\leq \frac16  \|\partial_t \uu_{k,\varepsilon}^n \|^2+
C \| \A \uu_{k,\varepsilon}^n\|^2
+ c_1 C\| \vp_{k,\varepsilon}^n\|_{H^2(\Omega)}
\| \nabla \uu_{k,\varepsilon}^n\| 
\|\A \uu_{k,\varepsilon}^n \|\\
&\leq \frac16  \|\partial_t \uu_{k,\varepsilon}^n \|^2+
c'_4 \Big( \| \A \uu_{k,\varepsilon}^n\|^2
+ \| \vp_{k,\varepsilon}^n\|_{H^2(\Omega)}^2
\| \nabla \uu_{k,\varepsilon}^n\|^2 \Big).
\end{align*}
By \eqref{dissest3app}, we have
\begin{align*}
(\mu_{k,\varepsilon}^n\nabla \vp_{k,\varepsilon}^n, 
\partial_t \uu_{k,\varepsilon}^n)
&\leq \| \mu_{k,\varepsilon}^n\|_{L^6(\Omega)}
\| \nabla \vp_{k,\varepsilon}^n\|_{\mathbf{L}^3(\Omega)}
\| \partial_t \uu_{k,\varepsilon}^n\| \\
&\leq  \frac16  \|\partial_t \uu_{k,\varepsilon}^n \|^2+
c'_5 \| \vp_{k,\varepsilon}^n\|_{H^2(\Omega)}^2 
(1+\| \nabla \mu_{k,\varepsilon}^n\|^2).
\end{align*}
Hence, we find 
\begin{align}
\label{ut}
\| \partial_t \uu_{k,\varepsilon}^n\|^2
&\leq c'_6 \Big( \| \A \uu_{k,\varepsilon}^n\|^2+
 \| \nabla  \uu_{k,\varepsilon}^n \|^4  
+ \| \vp_{k,\varepsilon}^n\|_{H^2(\Omega)}^2 (1+
\| \nabla \uu_{k,\varepsilon}^n\|^2 +\| \nabla \mu_{k,\varepsilon}^n\|^2) \Big).
\end{align}
Because of \eqref{sepvar} and \eqref{dissest1app}, we deduce that $g_i\in L^2(0,T)$, for all $i=1,...,n$, and $\uu_{k,\varepsilon}^n\in L^2(0,T;D(\A))$, which implies that $\A \uu_{k,\varepsilon}^n \in L^2(0,T;\H_\sigma)$.
By the theory of the Stokes operator (see Appendix \ref{stokesappendix}), there exists 
$p_{k,\varepsilon}^n\in L^2(0,T;V)$ 
such that
$-\Delta \uu_{k,\varepsilon}^n+\nabla p_{k,\varepsilon}^n= \A \uu_{k,\varepsilon}^n$ almost everywhere in $\Omega \times (0,T)$. 
In particular, we have
\begin{equation}
\label{pressure2}
\| p_{k,\varepsilon}^n\| \leq C \|\nabla \uu_{k,\varepsilon}^n \|^\frac12 \| \A \uu_{k,\varepsilon}^n\|^\frac12, \quad 
\| p_{k,\varepsilon}^n\|_{V}\leq C \|\A \uu_{k,\varepsilon}^n \|,
\end{equation}
where $C$ is independent of $k$, $\varepsilon$ and $n$.
Now we take $\vv=\A \uu_{k,\varepsilon}^n$ in \eqref{e1app} and we obtain
$$
\frac12 \ddt \| \nabla \uu_{k,\varepsilon}^n\|^2+ 
b(\uu_{k,\varepsilon}^n, \uu_{k,\varepsilon}^n,
\A \uu_{k,\varepsilon}^n)
-(\mathrm{div }\, (\nu(\vp_{k,\varepsilon}^n)
D\uu_{k,\varepsilon}^n),\A \uu_{k,\varepsilon}^n) 
= (\mu_{k,\varepsilon}^n \nabla \vp_{k,\varepsilon}^n,
 \A \uu_{k,\varepsilon}^n).
$$
We observe that
\begin{align*}
-(\mathrm{div }\, (\nu(\vp_{k,\varepsilon}^n)
D\uu_{k,\varepsilon}^n),\A \uu_{k,\varepsilon}^n)
&=-\frac12( \nu(\vp_{k,\varepsilon}^n)
 \Delta \uu_{k,\varepsilon}^n, 
\A \uu_{k,\varepsilon}^n )-
(\nu'(\vp_{k,\varepsilon}^n) 
 D \uu_{k,\varepsilon}^n \nabla \vp_{k,\varepsilon}^n, 
\A \uu_{k,\varepsilon}^n)\\
&=\frac12( \nu(\vp_{k,\varepsilon}^n) 
\A \uu_{k,\varepsilon}^n, \A \uu_{k,\varepsilon}^n)
-(\nu(\vp_{k,\varepsilon}^n) \nabla p_{k,\varepsilon}^n,
\A \uu_{k,\varepsilon}^n)\\
&\quad -(\nu'(\vp_{k,\varepsilon}^n) 
 D \uu_{k,\varepsilon}^n \nabla \vp_{k,\varepsilon}^n, 
\A \uu_{k,\varepsilon}^n)\\
&\geq \nu_\ast \|\A \uu_{k,\varepsilon}^n\|^2
+(\nu'(\vp_{k,\varepsilon}^n) \nabla \vp_{k,\varepsilon}^n p_{k,\varepsilon}^n,\A \uu_{k,\varepsilon}^n)\\
&\quad -(\nu'(\vp_{k,\varepsilon}^n)  D \uu_{k,\varepsilon}^n \nabla \vp_{k,\varepsilon}^n, 
\A \uu_{k,\varepsilon}^n).
\end{align*}
By \eqref{LADY}, \eqref{uniformapp} and \eqref{pressure2}, we have the following estimates
\begin{align*}
-(\nu'(\vp_{k,\varepsilon}^n) \nabla 
&\vp_{k,\varepsilon}^n p_{k,\varepsilon}^n,\A \uu_{k,\varepsilon}^n)
 +(\nu'(\vp_{k,\varepsilon}^n)  D \uu_{k,\varepsilon}^n \nabla \vp_{k,\varepsilon}^n, 
\A \uu_{k,\varepsilon}^n)\\
&\leq C \| \nabla \vp_{k,\varepsilon}^n\|_{\mathbf{L}^4(\Omega)} 
\Big( \| p_{k,\varepsilon}^n\|_{L^4(\Omega)}+ 
\| D \uu_{k,\varepsilon}^n\|_{\mathbf{L}^4(\Omega)} \Big)
\|\A \uu_{k,\varepsilon}^n \| \\
&\leq \sqrt{c_1} C \| \vp_{k,\varepsilon}^n\|_{H^2(\Omega)}^\frac12 
\Big( \| p_{k,\varepsilon}^n\|^\frac12 
\| p_{k,\varepsilon}^n\|_V^\frac12 +
\| \nabla \uu_{k,\varepsilon}^n\|^\frac12 \| \A \uu_{k,\varepsilon}^n\|^\frac12 \Big)
\|\A \uu_{k,\varepsilon}^n \| \\
&\leq \sqrt{c_1} C \| \vp_{k,\varepsilon}^n\|_{H^2(\Omega)}^\frac12
\Big( \|\nabla \uu_{k,\varepsilon}^n \|^\frac14 
\|\A \uu_{k,\varepsilon}^n\|^\frac34+ \| \nabla \uu_{k,\varepsilon}^n\|^\frac12 \| \A \uu_{k,\varepsilon}^n\|^\frac12 \Big) \| \A \uu_{k,\varepsilon}^n\|\\
&\leq \frac{\nu_\ast}{6} \|\A \uu_{k,\varepsilon}^n \|^2 +
c'_7 (1+\| \vp_{k,\varepsilon}^n\|_{H^2(\Omega)}^4)
\| \nabla \uu_{k,\varepsilon}^n\|^2,
\end{align*}
and
\begin{align*}
b(\uu_{k,\varepsilon}^n, \uu_{k,\varepsilon}^n,\A \uu_{k,\varepsilon}^n)
&\leq \| \uu_{k,\varepsilon}^n\|_{\mathbf{L}^4(\Omega)}
\| \nabla \uu_{k,\varepsilon}^n\|_{\mathbf{L}^4(\Omega)}
\| \A \uu_{k,\varepsilon}^n\| \\
&\leq \frac{\nu_\ast}{6} \|\A \uu_{k,\varepsilon}^n\|^2 + 
c'_8 \| \nabla  \uu_{k,\varepsilon}^n\|^4.
\end{align*}
Also, we have
\begin{align*}
(\mu_{k,\varepsilon}^n \nabla \vp_{k,\varepsilon}^n, \A \uu_{k,\varepsilon}^n)
&\leq \|\mu_{k,\varepsilon}^n\|_{L^6(\Omega)} 
\|\nabla \vp_{k,\varepsilon}^n\|_{\mathbf{L}^3(\Omega)} 
\| \A \uu_{k,\varepsilon}^n\|\\
&\leq  \frac{\nu_\ast}{6} \| \A \uu_{k,\varepsilon}^n\|^2
+c'_9 \| \vp_{k,\varepsilon}^n\|_{H^2(\Omega)}^2
 (1+\| \nabla \mu_{k,\varepsilon}^n\|^2).
\end{align*}
Hence, we are led to
\begin{align}
\label{diffineq2app}
\frac12 &\ddt \| \nabla \uu_{k,\varepsilon}^n\|^2+ 
\frac{\nu_\ast}{2}  \| \A \uu_{k,\varepsilon}^n\|^2\notag \\
&\leq c'_{10} \Big( \| \nabla  \uu_{k,\varepsilon}^n\|^4
+(1+\| \vp_{k,\varepsilon}^n\|_{H^2(\Omega)}^4)
\| \nabla \uu_{k,\varepsilon}^n\|^2 
+\|\vp_{k,\varepsilon}^n\|_{H^2(\Omega)}^2
 (1+\| \nabla \mu_{k,\varepsilon}^n\|^2) \Big).
\end{align}
Multiplying \eqref{ut} by $\varpi =\frac{\nu_\ast}{4c'_6}>0$ 
and summing up with \eqref{diffineq2app}, we arrive at 
\begin{align}
\label{diffineq2app2}
\frac12 &\ddt \| \nabla \uu_{k,\varepsilon}^n\|^2+ 
\frac{\nu_\ast}{4}  \| \A \uu_{k,\varepsilon}^n\|^2+
\varpi \| \partial_t \uu_{k,\varepsilon}^n\|^2 \notag \\
&\leq c'_{11} \Big( \| \nabla  \uu_{k,\varepsilon}^n\|^4+
(1+\| \vp_{k,\varepsilon}^n\|_{H^2(\Omega)}^4)
 (1+\| \nabla \uu_{k,\varepsilon}^n\|^2+\| \nabla \mu_{k,\varepsilon}^n\|^2) \Big).
\end{align}
Adding \eqref{diffineq1app} and \eqref{diffineq2app2}, we find the differential inequality
\begin{align}
\label{diffineq3app}
\ddt &\Lambda(\uu_{k,\varepsilon}^n,
\vp_{k,\varepsilon}^n)
+ \frac{\nu_\ast}{4} \|\A \uu_{k,\varepsilon}^n\|^2
+\frac{\varpi}{2} \|\partial_t \uu_{k,\varepsilon}^n\|^2
+ \frac14 \| \nabla \partial_t \vp_{k,\varepsilon}^n\|^2 \notag\\
&\leq  (\partial_t \uu_{k,\varepsilon}^n\cdot \nabla \vp_{k,\varepsilon}^n, \mu_{k,\varepsilon}^n)+ 
c'_{12} \| \nabla \uu_{k,\varepsilon}^n\|^4 \notag \\
&\quad + c'_{12} \Big( (1+\| \uu_{k,\varepsilon}^n\|_{\mathbf{L}^3(\Omega)}^2+ \| \vp_{k,\varepsilon}^n\|_{H^2(\Omega)}^4) (
1+ \| \nabla \uu_{k,\varepsilon}^n\|^2+\| \nabla \mu_{k,\varepsilon}^n\|^2 ) \Big),
\end{align}
where 
\begin{align}
\label{Lambdadef}
\Lambda(\uu_{k,\varepsilon}^n, \vp_{k,\varepsilon}^n)
=\frac12 \| \nabla \uu_{k,\varepsilon}^n\|^2
+\frac12 \| \nabla \mu_{k,\varepsilon}^n\|^2
+( \uu_{k,\varepsilon}^n 
\cdot \nabla \vp_{k,\varepsilon}^n,
\mu_{k,\varepsilon}^n).
\end{align}
We control the first term on the right-hand side of \eqref{diffineq3app} as follows
\begin{align*}
(\partial_t \uu_{k,\varepsilon}^n 
\cdot \nabla \vp_{k,\varepsilon}^n,
\mu_{k,\varepsilon}^n)
&\leq \| \partial_t \uu_{k,\varepsilon}^n\| 
\| \nabla \vp_{k,\varepsilon}^n\|_{\mathbf{L}^3(\Omega)} 
\| \mu_{k,\varepsilon}^n\|_{L^6(\Omega)}\\
&\leq \frac{	\varpi}{4} \| \partial_t \uu_{k,\varepsilon}^n\|^2
+ C \| \nabla \vp_{k,\varepsilon}^n\|_{\mathbf{L}^3(\Omega)}^2 
\| \mu_{k,\varepsilon}^n\|_V^2\\
&\leq \frac{\varpi}{4} \| \partial_t \uu_{k,\varepsilon}^n\|^2
+ c'_{13} \| \vp_{k,\varepsilon}^n\|_{H^2(\Omega)}^2
(1+\| \nabla \mu_{k,\varepsilon}^n\|^2).
\end{align*}
Then, we arrive at
\begin{align}
\label{diffineq4app}
\ddt &\Lambda(\uu_{k,\varepsilon}^n,
\vp_{k,\varepsilon}^n)
+ \frac{\nu_\ast}{4} \|\A \uu_{k,\varepsilon}^n\|^2
+\frac{\varpi}{4} \|\partial_t \uu_{k,\varepsilon}^n\|^2
+ \frac14 \| \nabla \partial_t \vp_{k,\varepsilon}^n\|^2 \notag\\
&\leq   
c'_{14} \Big( \| \nabla \uu_{k,\varepsilon}^n\|^4 + (1+\| \uu_{k,\varepsilon}^n\|_{\mathbf{L}^3(\Omega)}^2+ \| \vp_{k,\varepsilon}^n\|_{H^2(\Omega)}^4) (
1+ \| \nabla \uu_{k,\varepsilon}^n\|^2+\| \nabla \mu_{k,\varepsilon}^n\|^2 ) \Big).
\end{align}
Now we show that $\Lambda(\uu_{k,\varepsilon}^n, 
\vp_{k,\varepsilon}^n)$ is bounded from below. 
By using \eqref{LADY} and exploiting \eqref{dissest1app}-\eqref{dissest3app}, we have
\begin{align*}
( \uu_{k,\varepsilon}^n\cdot 
\nabla \vp_{k,\varepsilon}^n,\mu_{k,\varepsilon}^n) 
&\leq \|\uu_{k,\varepsilon}^n\|_{\mathbf{L}^4(\Omega)} 
\| \nabla \vp_{k,\varepsilon}^n\| 
\| \mu_{k,\varepsilon}^n\|_{L^4(\Omega)}\\
&\leq c_1 C \| \uu_{k,\varepsilon}^n\|^{\frac12}
 \| \nabla \uu_{k,\varepsilon}^n\|^{\frac12}\
\| \mu_{k,\varepsilon}^n\|_V\\
&\leq \frac14 \| \nabla \uu_{k,\varepsilon}^n\|^2+ \frac14\| \nabla \mu_{k,\varepsilon}^n\|^2+c'_{15}.
\end{align*}
Hence, we infer that
\begin{equation}
\label{lambdabelow}
\Lambda(\uu_{k,\varepsilon}^n, \vp_{k,\varepsilon}^n) \geq \frac14 \| \nabla \uu_{k,\varepsilon}^n\|^2
+ \frac14 \| \nabla \mu_{k,\varepsilon}^n\|^2-c'_{15}.
\end{equation}
Moreover, it is easily seen that
\begin{equation}
\label{lambdaabove}
\Lambda(\uu_{k,\varepsilon}^n, \vp_{k,\varepsilon}^n) 
\leq c'_{16} \Big( 1+ \| \nabla \uu_{k,\varepsilon}^n\|^2
+ \| \nabla \mu_{k,\varepsilon}^n\|^2\Big).
\end{equation}
In summary, exploiting \eqref{vpapprH2} and the Sobolev embedding $V\hookrightarrow L^3(\Omega)$, we are led to rewrite \eqref{diffineq4app} as
\begin{align}
\label{diffineq4}
\ddt \Lambda(\uu_{k,\varepsilon}^n, \vp_{k,\varepsilon}^n) &+ \overline{\nu} \Big(\|\A \uu_{k,\varepsilon}^n\|^2
+ \| \partial_t \uu_{k,\varepsilon}^n\|^2
+ \| \nabla \partial_t \vp_{k,\varepsilon}^n\|^2 \Big)
\leq  c'_{17} \Big( 1+\Lambda^2(\uu_{k,\varepsilon}^n, \vp_{k,\varepsilon}^n)\Big),
\end{align}
where $\overline{\nu}=\frac14 \min \lbrace 1, \nu_\ast, \varpi \rbrace$.
Owing to \eqref{Psieps-prop}, \eqref{dissest1app} and \eqref{lambdaabove}, we infer that
$$
\int_0^T \Lambda(\uu_{k,\varepsilon}^n(s), \vp_{k,\varepsilon}^n(s)) \, \d s \leq c'_{18}.
$$
An application of the Gronwall lemma to \eqref{diffineq4} implies that 
\begin{equation}
\label{Lambdaest1}
\Lambda(\uu_{k,\varepsilon}^n(t), 
\vp_{k,\varepsilon}^n(t))
\leq  \Lambda(\uu_{k,\varepsilon}^n(0), 
\vp_{k,\varepsilon}^n(0)) \mathrm{e}^{c'_{18}}+ c'_{17}\mathrm{e}^{c'_{18}} T, \quad \forall \, t \in [0,T],
\end{equation}
In order to find a uniform control of the right-hand side of \eqref{Lambdaest1}, by using the Sobolev embedding $V\hookrightarrow L^6(\Omega)$, \eqref{psikH2} and \eqref{psiH1}, we obtain
\begin{align*}
\Lambda(\uu_{k,\varepsilon}^n(0), 
\vp_{k,\varepsilon}^n(0))&= 
\Lambda( P_n \uu_0, \Pi_n\vp_{0,k})\\
&= \frac12 \| \nabla P_n \uu_0\|^2
+\frac12 \| \nabla \mu_{k,\varepsilon}^n(0)\|^2
+( P_n \uu_0 
\cdot \nabla \Pi_n \vp_{0,k},
\mu_{k,\varepsilon}^n(0))\\
&\leq \frac12 \|\nabla \uu_0\|^2
+\frac12 \| \nabla \mu_{k,\varepsilon}^n(0)\|^2
+ \| P_n \uu_0\|_{\mathbf{L}^3(\Omega)} 
\| \nabla \Pi_n\vp_{0,k}\|
\|\mu_{k,\varepsilon}^n(0)\|_{L^6(\Omega)}\\
&\leq \|\nabla \uu_0\|^2+ C \big( 1+\| \vp_{0,k}\|_V^2\big)\|\mu_{k,\varepsilon}^n(0)\|_V^2\\
&\leq \|\nabla \uu_0\|^2+ C \big( 1+\| \vp_{0}\|_V^2\big)\|\mu_{k,\varepsilon}^n(0)\|_V^2.
\end{align*}
In light of \eqref{psikH2}, \eqref{psiH1}, \eqref{muk2H1} and \eqref{Approxinitcond}, we find 
\begin{align}
\| \mu_{k,\varepsilon}^n(0)&\|_V=
\| \Pi_n (-\Delta \vp_{k,\varepsilon}^n(0)+
\Psi'_\varepsilon(\vp_{k,\varepsilon}^n(0)))\|_V \notag\\
&\leq  \| -\Delta \vp_{k,\varepsilon}^n(0)+
F'_\varepsilon(\vp_{k,\varepsilon}^n(0))\|_V + \theta_0 \|\vp_{k,\varepsilon}^n(0) \|_V \notag \\
&\leq  \| -\Delta \vp_{k,\varepsilon}^n(0)
+F'_\varepsilon(\vp_{k,\varepsilon}^n(0))+ 
\Delta \vp_{0,k}-F_\varepsilon'(\vp_{0,k})\|_V
+C ( \| \tilde{\mu}_{0,k}\|_V+ \| \vp_{0}\|_V ) \notag \\
&\leq \| \vp_{k,\varepsilon}^n(0)-
\vp_{0,k}\|_{H^3(\Omega)} + 
 \| F'_\varepsilon(\vp_{k,\varepsilon}^n(0))-F'_\varepsilon(\vp_{0,k})\|_V+ C( \| \tilde{\mu}_{0}\|_V+ \| \vp_{0}\|_V ) \notag \\
&=  \| \Pi_n \vp_{0,k}-
\vp_{0,k}\|_{H^3(\Omega)} + 
\| F'_\varepsilon(\Pi_n \vp_{0,k})-F'_\varepsilon(\vp_{0,k})\|_V+ C( \| \mu_{0}\|_V+ \| \vp_{0}\|_V ). \label{estdatn1}
\end{align}
Recalling the bounds \eqref{psikinf} and \eqref{Linfdatvare} and the relation $F(z)=F_\varepsilon(z)$, for all $z\in [-1+\overline{\varepsilon},1-\overline{\varepsilon}]$ (cf. $0<\varepsilon< \overline{\varepsilon}$), we deduce that
\begin{align}
\| F'_\varepsilon(\Pi_n \vp_{0,k})&-F'_\varepsilon(\vp_{0,k})\|_V \notag \\
&\leq  \| F'_\varepsilon(\Pi_n \vp_{0,k})-F'_\varepsilon(\vp_{0,k})\| +\|F''_\varepsilon(\Pi_n \vp_{0,k}) 
\nabla(\Pi_n \vp_{0,k} -\vp_{0,k})\| \notag\\
&\quad +\| (F''_\varepsilon(\Pi_n \vp_{0,k})-
F''_\varepsilon(\vp_{0,k}))\nabla \vp_{0,k}\|\notag \\
&\leq C \Big(\max_{z\in [-1+\overline{\varepsilon},1-\overline{\varepsilon}]} |F''(z)|+ \max_{z\in [-1+\overline{\varepsilon},1-\overline{\varepsilon}]} |F'''(z)|\Big) \| \Pi_n \vp_{0,k}-
\vp_{0,k}\|_V. \label{estdatn2}
\end{align}
We notice that the quantity between brackets in \eqref{estdatn2} is finite since $F\in \mathcal{C}^3(-1,1)$, and it only depends on $k$ (cf. definition of $\overline{\varepsilon}$). Let us now recall that
$\Pi_n \vp_{0,k} \rightarrow \vp_{0,k}$ in 
$H^3(\Omega)$ as $n\rightarrow \infty$. Thus, we infer from  \eqref{estdatn1}-\eqref{estdatn2} that,
for any fixed $k>\overline{k}$ (and $\varepsilon\in (0,\overline{\varepsilon})$), 
there exists $\overline{\overline{n}}>\overline{n}$ (cf. \eqref{Approxinitcond}) such that 
\begin{equation}
\label{estdatn3}
\| \mu_{k,\varepsilon}^n(0)\|_V \leq 
C(1+ \| \mu_{0}\|_V+ \| \vp_{0}\|_V ), \quad \forall \, n >\overline{\overline{n}},
\end{equation}
where $C$ is independent of $k$, $n$ and $\varepsilon$.
Finally, for any fixed $k>\overline{k}$, $\varepsilon \in (0,\overline{\varepsilon})$ and $n>\overline{\overline{n}}$ (where $\overline{\varepsilon}$ and $\overline{\overline{n}}$ depends on $k$),
we infer from \eqref{Lambdaest1} and \eqref{estdatn3} that
$$
\Lambda(\uu_{k,\varepsilon}^n(t), 
\vp_{k,\varepsilon}^n(t))
\leq C(1+\| \mu_0\|_V+\| \vp_0\|_V)\mathrm{e}^{c'_{18}}+ 
c'_{17}\mathrm{e}^{c'_{18}} T, \quad \forall \, t\in [0,T].
$$
In view of \eqref{lambdabelow}, we have
\begin{equation}
\label{unif1}
\sup_{t\in [0,T]} \| \nabla \uu_{k,\varepsilon}^n(t)\|
+\sup_{t\in [0,T]}\| \nabla \mu_{k,\varepsilon}^n(t)\|\leq \overline{C}_1,
\end{equation}
where $\overline{C}_1$ is a positive constant, which depends on $T$ and $\E(\uu_0,\vp_0)$, $\| \uu_0\|_{\V_\sigma}$ and $\| \mu_0\|_V$, but is independent of $k$, $n$ and $\varepsilon$.
Moreover, an integration in time of \eqref{diffineq4} 
on the time interval $[0,T]$ yields
\begin{equation}
\label{unif2}
\int_0^T 
\Big( \|\A \uu_{k,\varepsilon}^n(s)\|^2
+\| \partial_t \uu_{k,\varepsilon}^n(s)\|^2
+ \| \nabla \partial_t \vp_{k,\varepsilon}^n(s)\|^2\Big)
\, \d s \leq \overline{C}_2,
\end{equation}
where $\overline{C}_2$ is a positive constant depending on $T$ and on the initial datum, but independent of $k$, $\varepsilon$ and $n$.
\smallskip

\noindent
\textbf{6. Passage to the limit.} 
Thanks to the analysis performed in Step $5$, 
for any fixed $k>\overline{k}$, $\varepsilon \in (0,\overline{\varepsilon})$ and $n>\overline{\overline{n}}$, 
we deduce from \eqref{unif1} and \eqref{unif2} that
\begin{align*}
& \uu_{k,\varepsilon}^n \ \text{is uniformly bounded in }L^{\infty}(0,T;\V_\sigma)\cap L^2(0,T;\W_\sigma)\cap
H^1(0,T;\H_\sigma),\\
& \varphi_{k,\varepsilon}^n \ 
\text{is uniformly bounded in } L^\infty(0,T;H^2(\Omega))
\cap H^1(0,T;V),\\
& \mu_{k,\varepsilon}^n \ \text{is uniformly bounded in }L^\infty(0,T;V).
\end{align*}
By a standard compactness method,
we are in position to pass to the limit first as $n\rightarrow\infty$, then as $\varepsilon \rightarrow 0$ and, finally, as $k\rightarrow \infty$. 
As a result, we obtain the existence of a pair $(\uu,\vp)$ such that
\begin{align*}
&\uu \in  L^{\infty}(0,T;\V_\sigma)\cap L^{2}(0,T; \W_\sigma) \cap H^{1}(0,T; \H_\sigma),\\
&\vp \in  L^\infty(0,T; H^2(\Omega))\cap  H^1(0,T; V),\\
&\vp \in L^{\infty}(\Omega\times (0,T)),\quad\text{with}\quad |\varphi(x,t)|<1 
\ \text{a.e. }(x,t)\in \Omega\times (0,T),
\end{align*}
which satisfies \eqref{e1} and \eqref{e2}, 
where $\mu=-\Delta \varphi +\Psi'(\varphi)\in L^\infty(0,T;V)$. Moreover, $\partial_\n \varphi=0$ almost everywhere 
on $\partial\Omega\times(0,T)$, $\uu(\cdot,0)=\uu_0$ and 
$\varphi(\cdot,0)=\varphi_0$ in $\Omega$.
Since $\partial_t \vp + \uu \cdot \nabla \vp$ belongs to $L^2(0,T;V)$ owing to the above regularity properties, we infer from the classical regularity theory of the homogeneous Neumann
operator that $\mu \in L^2(0,T;H^3(\Omega))$, $\partial_n \mu=0$ almost everywhere on $\partial\Omega\times(0,T)$ and $\partial_t \vp+ \uu \cdot \nabla \vp= \Delta \mu$ holds almost everywhere in $\Omega \times (0,T)$. 
Finally, we can recover the pressure $\pi$ arguing as in \cite[Propositions 1.1 and 1.2, Chapter III]{Temam}. In particular, it is possible to show that there exists
$\pi \in L^2(0,T;V)$ such that
$\partial_t \uu+(\uu \cdot \nabla) \uu-\mathrm{div }\, (\nu(\vp)D\uu)
+\nabla \pi= \mu \nabla \vp$ holds almost everywhere in $\Omega \times (0,T)$.
\smallskip

\noindent
\textbf{7. Further regularity properties.}
From the regularity $\mu \in L^\infty(0,T;V)$, an application of 
Theorem \ref{ell2} entails that 
$\vp\in L^\infty(0,T;W^{2,p}(\Omega))$ and 
$F'(\vp)\in L^\infty(0,T;L^p(\Omega))$, for any $2\leq p<\infty$.
Furthermore, thanks to the growth condition \eqref{Fsec}, we also
deduce that 
$F''(\varphi)\in L^\infty(0,T;L^p(\Omega))$, for any $p\in (2,\infty)$.
Next, as a consequence, we prove that $\partial_t \mu$ exists and belongs to $L^2(0,T;V')$.
To this aim, given $h>0$, we denote the difference quotient of a function $f$ by
$\partial_t^h f=\frac{1}{h}\big(f(t+h)-f(t)\big)$.
For any $v \in V$ with $\| v\|_V\leq 1$, by using the boundary condition on $\vp$, we observe that
$
(\partial_t^h \mu, v)
= (\nabla \partial_t^h \vp, \nabla v)+(\partial_t^h F'(\vp),v)-\theta_0 (\partial_t^h \vp,v).
$
Since $F''$ is convex, we find the control
\begin{align}
(\partial_t^h F'(\vp),v)
&\leq \Big\| \int_0^1 \Big( s F''(\varphi(\cdot+h))+(1-s)F''(\varphi)\Big)\, \d s \Big\|_{L^3(\Omega)}
\|\partial_t^h\varphi \| \| v\|_{L^6(\Omega)} \notag \\
&\leq C \Big( \| F''(\vp(\cdot +h))\|_{L^3(\Omega)}+ \| F''(\vp)\|_{L^3(\Omega)}
\Big) \|\partial_t^h\varphi\|.
\label{Fdiff}
\end{align}
Recalling that $\partial_t^h \vp \rightarrow \partial_t \vp$ in $L^2(0,T;V)$ and $F''(\vp)\in L^\infty(0,T;L^3(\Omega))$, there exists a positive constant $\overline{C}_3$, independent of $h$, such that 
$\| \partial_t^h \mu\|_{L^2(0,T;V')}\leq \overline{C}_3$. This implies that $\partial_t \mu \in L^2(0,T;V')$.
In particular, we deduce that $\mu \in \mathcal{C}([0,T],V)$.
\smallskip

\smallskip

\noindent
\textbf{8. Uniqueness and Continuous dependence.} 
The uniqueness of strong solutions is an immediate consequence of Theorem \ref{weak-uniqueness}. We conclude the proof by showing a continuous dependence estimate with respect to the initial conditions in  higher-order norms than the dual norms employed in Theorem \ref{weak-uniqueness}. 
We define $\uu=\uu_1-\uu_2$ and $\vp=\vp_1-\vp_2$, where 
$(\uu_1,\vp_1)$ and $(\uu_2,\vp_2)$ are two strong solutions departing from $(\uu_{01},\vp_{01})$ and $(\uu_{02},\vp_{02})$ which satisfy
$\uu_{0i}\in \V_\sigma$ and $\vp_{0i} \in H^2(\Omega)$ such that 
$\| \vp_{0i}\|_{L^\infty(\Omega)}\leq 1$, $|\overline{\vp}_{0i}|<1$, $\mu_{0i}=-\Delta \vp_{0i}+\Psi'(\vp_{0i}) \in V$ and 
$\partial_\n \vp_{0i}=0$ on $\partial \Omega$. 
We take $\vv=\uu$ and $v=\vp$ in \eqref{e1u} and \eqref{e2u}, respectively. Adding the resulting equalities, we find
$$
\ddt \mathcal{H}_1 + (\nu(\vp_1)D\uu, D\uu)
+(\nabla \mu, \nabla \vp)= \sum_{k=1}^3 \mathcal{J}_k,
$$
having set 
\begin{align*}
&\mathcal{H}_1=\frac12 \| \uu\|^2+ \frac12 \| \vp\|^2,\\
&\mathcal{J}_1=-((\uu\cdot \nabla)\uu_2,\uu)- ((\nu(\vp_1)-\nu(\vp_2))D \uu_2, \nabla \uu),\\
&\mathcal{J}_2= (\nabla \vp_1 \otimes \nabla \vp,\nabla \uu)+ (\nabla \vp\otimes \nabla \vp_2,\nabla \uu),\\
&\mathcal{J}_3=(\vp \uu_1, \nabla \vp)+ (\vp_2 \uu, \nabla \vp).
\end{align*}
In light of the regularity of strong solutions, there exists a positive constant $C_0$ such that
\begin{equation}
\label{regstronguniq}
\| \uu_i\|_{L^\infty(0,T;\mathbf{L}^3(\Omega))}+\| \vp\|_{L^\infty(0,T;W^{2,3}(\Omega))}+\| \Psi(\vp_i)\|_{L^\infty(0,T;L^3(\Omega))}\leq C_0.
\end{equation}
In the sequel the positive constant $C_i$, $i\in \mathbb{N}$ depends on $\nu_\ast$, $\nu'$, $C_0$ and the constants appearing in embedding results.
Due to the homogeneous Neumann boundary condition, we also recall the basic inequality
\begin{equation}
\label{basicineq}
\| \vp\|_V^2 \leq \| \Delta \vp\| \| \vp\|+ \| \vp\|^2.
\end{equation}
Integrating by parts and using the embedding $V\hookrightarrow L^6(\Omega)$, together with \eqref{regstronguniq} and \eqref{basicineq}, we observe that
\begin{align*}
(\nabla \mu,\nabla \vp)
&\geq \| \Delta \vp\|^2- \big( \| \Psi''(\vp_1)\|_{L^3(\Omega)}+ \| \Psi''(\vp_2)\|_{L^3(\Omega)}\big) \|\vp\|_{L^6(\Omega)}\| \Delta \vp\| \\
&\geq \frac12 \| \Delta \vp\|^2- C_1 \| \vp\|_V^2  \\
&\geq \frac14 \| \Delta \vp\|^2- C_2 \| \vp\|^2. 
\end{align*}
Due to the Korn inequality and the above estimate, we obtain
$$
\ddt \mathcal{H}_1 + \nu_\ast \| \nabla \uu\|^2+ \frac14 \| \Delta \vp\|^2
\leq C_2 \| \vp\|^2 +\sum_{k=1}^3 \mathcal{J}_k.
$$
We now address the terms $\mathcal{J}_k$. By using \eqref{basicineq}, we have
\begin{align*}
\mathcal{J}_1&\leq \|\uu \| \| \nabla \uu_2\|_{\mathbf{L}^3(\Omega)} \| \uu\|_{\mathbf{L}^6(\Omega)}+ C \| \vp\|_{L^6(\Omega)} \| D \uu_2\|_{\mathbf{L}^3(\Omega)} \| \nabla \uu\|\\
&\leq \frac{\nu_\ast}{4} \|\nabla \uu \|^2+ C \| \nabla \uu_2\|_{\mathbf{L}^3(\Omega)}^2 \| \uu\|^2
+C\| D \uu_2\|_{\mathbf{L}^3(\Omega)}^2 \| \vp\|_V^2
,\\
&\leq \frac{\nu_\ast}{4} \|\nabla \uu \|^2+ \frac{1}{24} \| \Delta \vp\|^2+C_3 \Big( \| \nabla \uu_2\|_{\mathbf{L}^3(\Omega)}^2 \| \uu\|^2+ \| D \uu_2\|_{\mathbf{L}^3(\Omega)}^4 \| \vp\|^2\Big).
\end{align*}
By \eqref{regstronguniq} and \eqref{basicineq} and the embedding $W^{2,3}(\Omega)\hookrightarrow W^{1,\infty}(\Omega)$ valid in dimension two, we obtain
\begin{align*}
\mathcal{J}_2&\leq \big(\|\nabla \vp_1 \|_{\mathbf{L}^\infty(\Omega)}+\|\nabla \vp_2 \|_{\mathbf{L}^\infty(\Omega)}\big) \|\nabla \vp\| \| \nabla \uu\|\\
&\leq \frac{\nu_\ast}{4} \|\nabla \uu \|^2+ C_4\| \nabla \vp\|^2\\
&\leq  \frac{\nu_\ast}{4} \|\nabla \uu \|^2
+ \frac{1}{24} \|\Delta \vp \|^2+C_5\| \vp\|^2,
\end{align*}
and
\begin{align*}
\mathcal{J}_3&\leq \| \vp\|_{L^6(\Omega)} \| \uu_1\|_{\mathbf{L}^3(\Omega)} \| \nabla \vp\|+ \| \uu\| \| \nabla \vp\| \\
&\leq C_6 \|\vp\|_V^2 + \| \uu\|^2\\
&\leq \frac{1}{24}\| \Delta \vp\|^2+C_7 \Big( \| \vp\|^2 + \| \uu\|^2\Big).
\end{align*}
In view of the above estimates, we end up with the following differential inequality
$$
\ddt \mathcal{H}_1+ \frac{\nu_\ast}{2}\| \nabla \uu\|^2+
\frac{1}{8}\| \Delta \vp\|^2\leq C_8 \Big( 1+ \| \uu_2\|_{\mathbf{W}^{1,3}(\Omega)}^4\Big)\mathcal{H}_1.
$$
Therefore, since $\uu_2\in L^4(0,T; \mathbf{W}^{1,3}(\Omega))$, an application of the Gronwall lemma implies
the desired stability inequality \eqref{uniqhigh}. 
\end{proof}


By virtue of the energy identity (cf. \eqref{EI}) and the global well-posedness of the strong solutions, we can prove that the (unique) weak solution regularizes instantaneously.
That is, the weak solution is indeed a strong solution on
$\Omega \times (\tau,\infty)$, for any $\tau>0$. 

\begin{theorem}
\label{regularity}
Let $d=2$, $R>0$, $m\in (-1,1)$ and $\tau>0$ be given.
Assume that $(\uu_0,\vp_0)$ is an initial datum such that
$\E(\uu_0,\varphi_0)\leq R$, $\| \vp_0\|_{L^\infty(\Omega)}\leq 1$ and $\overline{\varphi}_0=m$,
and $(\uu,\vp)$ is the weak solution departing from $(\uu_0,\vp_0)$.
Then, there exist two positive constants $M_1=M_1(R,m,\tau)$ and $M_2=M_2(R,m,\tau)$, independent of the specific datum $(\uu_0,\vp_0)$, such that
\begin{equation}
\label{reg1}
 \sup_{t\geq \tau}\|\uu(t)\|_{\V_\sigma}+
  \sup_{t\geq \tau} \|\mu(t)\|_{V}\leq M_1,
\end{equation}
and
\begin{equation}
\label{reg2}
\| \uu\|_{L^2(t,t+1;\W_\sigma)}
+\| \partial_t \uu\|_{L^2(t,t+1;\H_\sigma)}
+\| \partial_t \vp\|_{L^2(t,t+1;V)}\leq M_2, \quad \forall \, t\geq \tau.
\end{equation}
In addition, for any $p\geq 2$, there exists a positive constant $M_3=M_3(R,m,\tau,p)$ such that
\begin{equation}
\label{reg3}
\sup_{t\geq \tau}\| \vp(t)\|_{W^{2,p}(\Omega)}+ \| F''(\vp)\|_{L^\infty(\tau,\infty;L^p(\Omega))}\leq M_3.
\end{equation}
\end{theorem}

\begin{proof}
Let $(\uu,\vp)$ be the global weak solution with initial condition $(\uu_0,\vp_0)$ given by Theorem \eqref{existence}.
Due to \eqref{EI}, for any $\tau>0$, we infer from \eqref{EI} that there exists $\tau_0\in (0,\tau)$ such that $(\uu(\tau_0),\vp(\tau_0))$ satisfies the assumptions of Theorem \ref{strong} and 
\begin{equation}
\label{estinitdatum}
\E(\uu(\tau_0), \vp(\tau_0))\leq R, \quad \overline{\vp}(\tau_0)=m.
\end{equation}
Taking $(\uu(\tau_0),\vp(\tau_0))$ as initial datum, we have a global strong solution on the time interval $[\tau_0,\infty)$, which coincides with the weak solution due to Theorem \ref{weak-uniqueness}. 
Now, in order to show the uniform estimates \eqref{reg1}-\eqref{reg3}, we consider the approximating solutions $(\uu_{k,\varepsilon}^n,\vp_{k,\varepsilon}^n)$ constructed in the proof of Theorem \ref{strong} on the time interval $[\tau_0,\infty)$ corresponding to the initial datum $(\uu(\tau_0),\vp(\tau_0))$.
Thanks to \eqref{E0} and \eqref{eneqappr}, we have
\begin{equation}
\label{estregtau1}
\E_\varepsilon(\uu_{k,\varepsilon}^n(t),\vp_{k,\varepsilon}^n(t)) 
+  \int_t^{t+1} \Big( \nu_\ast \| \nabla \uu_{k,\varepsilon}^n(s)\|^2
+\| \nabla \mu_{k,\varepsilon}^n(s)\|^2 \Big)
\, \d s \leq \tilde{c}_0, \quad \forall \, t \geq \tau_0,
\end{equation}
where $\tilde{c}_0$ depends on $R$, but is independent of $t$. 
Then, following line by line Steps $4$ and $5$ in the proof of Theorem \ref{strong}, we deduce the differential inequality (cf. \eqref{diffineq4})
\begin{equation}
\label{estregtau2}
\ddt \Lambda(\uu_{k,\varepsilon}^n, \vp_{k,\varepsilon}^n) + \overline{\nu} \Big(\|\A \uu_{k,\varepsilon}^n\|^2
+ \| \partial_t \uu_{k,\varepsilon}^n\|^2
+ \| \nabla \partial_t \vp_{k,\varepsilon}^n\|^2 \Big)
\leq  \tilde{c}_1 \Big( 1+ \Lambda^2(\uu_{k,\varepsilon}^n, \vp_{k,\varepsilon}^n)\Big),
\end{equation} 
where $\Lambda(\uu_{k,\varepsilon}^n, \vp_{k,\varepsilon}^n)$ is defined in $\eqref{Lambdadef}$. Here, the positive constants $\overline{\nu}$ and $\tilde{c}_1$ depend on $R$, $m$ and the other parameters of the system, but are independent of $k$, $\varepsilon$ and $n$. By \eqref{ipo-nu} and \eqref{estregtau1}, we notice that
$$
\int_t^{t+1} \Lambda(\uu_{k,\varepsilon}^n(s), \vp_{k,\varepsilon}^n)(s) \, \d s\leq \tilde{c}_2, \quad \forall \, t \geq \tau_0. 
$$
Hence, an application of the uniform Gronwall lemma (see \cite[Chapter III, Lemma 1.1]{T}) to \eqref{estregtau2} with $r=\tau-\tau_0$ entails
$$
\| \uu_{k,\varepsilon}^n(t)\|_{\V_\sigma}+ \|\mu_{k,\varepsilon}^n(t))\|_{V}
\leq M_1, \quad \forall \,  t \geq \tau,
$$
where $M_1$ depends on $R$, $m$ and $\tau$. In particular, $M_1$ does not depend on $(\uu(\tau_0),\vp(\tau_0))$. In addition, integrating in time \eqref{estregtau2} on $(t,t+1)$, for any $t\geq \tau$, we are lead to
$$
\| \uu_{k,\varepsilon}^n\|_{L^2(t,t+1;\W_\sigma)}
+ \| \partial_t \uu_{k,\varepsilon}^n\|_{L^2(t,t+1;\H_\sigma)}
+ \| \partial_t \vp_{k,\varepsilon}^n\|_{L^2(t,t+1;V)}
\leq M_2, \quad \forall \, t \geq \tau.
$$
At this stage, passing to the limit in $k$, $\varepsilon$ and $n$ as in the proof of Theorem \ref{strong}, and using the regularity in time of the strong solutions, the estimates \eqref{reg1} and \eqref{reg2} easily follow. In turn, we also infer the estimate \eqref{reg3} from 
Theorem \ref{ell2}.
\end{proof}

As a consequence of Proposition \ref{prop-uniq} and Theorem \ref{regularity}, we deduce the continuous dependence of weak solutions with respect to the initial data in the energy space.

\begin{proposition}
\label{dep-energy-space}
Let $d=2$. Assume that a sequence of initial data $(\uu_{0n},\vp_{0n})$ and $(\uu_{0},\vp_{0})$ are given such that $\uu_{0n}\in \H_\sigma$, $\vp_{0n}\in V$, $\| \vp_{0n}\|_{L^\infty(\Omega)}\leq 1$ and $\overline{\vp}_{0n}=m$ with $m \in (-1,1)$ for all $n$, and $(\uu_{0n},\vp_{0n})$ converges to $(\uu_0,\vp_0)$ in $\H_\sigma \times V$. Consider the solutions $(\uu_n,\vp_n)$, $(\uu,\vp)$ to \eqref{system}-\eqref{bcic} with initial data $(\uu_{0n},\vp_{0n})$ and $(\uu_{0},\vp_{0})$, respectively. Then, for any $t>0$, 
$(\uu_n(t),\vp_n(t))$ converges to $(\uu(t),\vp(t))$ in $\H_\sigma \times V$.
\end{proposition}

\begin{proof}
Let us fix $t>0$. By assumption there exists $N_0>0$ such that 
$\| \uu_{0n}\|+\| \vp_{0n}\|_V \leq N_0$ and $\| \uu_{0}\|+\| \vp_{0}\|_V\leq N_0$. By Theorem \ref{regularity} (with $\tau=\frac{t}{2}$) there exists $N_1$ depending only on $N_0, m, t$ such that $\| \uu_{n}(t)\|_{\V_\sigma}+\| \vp_{n}(t)\|_{H^2(\Omega)}\leq N_1$. Obviously, the same control in $\V_\sigma \times H^2(\Omega)$ holds for $(\uu,\vp)$. 
By Proposition \ref{prop-uniq} we infer that there exists $N_2$ depending on $N_0$ and $m$ such that
$$
\| \uu_n (t) -\uu (t)\|_{\sharp}^2+ 
\| \vp_n(t)-\vp(t)\|_{\ast}^2 
\leq N_2 \Big( \frac{\| \uu_{0n} -\uu_0\|_{\sharp}^2+ 
\| \vp_{0n}-\vp_0\|_{\ast}^2}{N_2}\Big)^{\mathrm{e}^{-\int_0^t \mathcal{Y}(s)} \, \d s}, 
$$
where
$$
\mathcal{Y}(t)=N_2 \Big(1+ \| \uu_{n}(t)\|_{\V_\sigma}^2+
\| \uu(t)\|_{\V_\sigma}^2+ \| \nabla \vp_n(t)\|_{\mathbf{L}^\infty(\Omega)}^2+ \| \nabla \vp(t)\|_{\mathbf{L}^\infty(\Omega)}^2+
\| \vp_{n}(t)\|_{H^2(\Omega)}^2.
$$
Noticing that $\mathcal{Y}(t)\geq N_2$, assuming that $\| \uu_{0n} -\uu_0\|_{\sharp}^2+ 
\| \vp_{0n}-\vp_0\|_{\ast}^2\leq 1$, by interpolation we have
\begin{align*}
&\| \uu_n (t) -\uu (t)\|+ \| \vp_n(t)-\vp(t)\|_V \\
&\leq C  \Big( \| \uu_n (t) -\uu (t)\|_{\sharp}^\frac{1}{2}+ 
\| \vp_n(t)-\vp(t)\|_{\ast}^{\frac14} \Big) 
\Big( \| \uu_n (t) -\uu (t)\|_{\V_\sigma}^\frac{1}{2}+ 
\| \vp_n(t)-\vp(t)\|_{H^2(\Omega)}^{\frac34} \Big)\\
&\leq C(N_1^\frac12+N_1^\frac34)
(N_2^\frac14+N_2^\frac18)
\Big( \frac{\| \uu_{0n} -\uu_0\|_{\sharp}^2+ 
\| \vp_{0n}-\vp_0\|_{\ast}^2}{N_2}\Big)^{\frac14 \mathrm{e}^{-N_2 t}}.
\end{align*}
The above inequality implies the desired conclusion.
\end{proof}

Our next result concerns the propagation of regularity for any weak solution and the validity of the instantaneous separation property from the pure concentrations (i.e. $\pm 1$) in dimension two.
This is possible due to a suitable estimate of $\Psi''(\vp)$ in $L^p$ spaces, which allows us to show further a-priori higher-order Sobolev estimates.

\begin{theorem}
\label{separazione}
Let $d=2$, $R>0$, $m\in (-1,1)$ and $\tau>0$ be given.
Assume that $(\uu_0,\vp_0)$ is an initial datum such that
$\E(\uu_0,\varphi_0)\leq R$, $\| \vp_0\|_{L^\infty(\Omega)}\leq 1$ and $\overline{\varphi}_0=m$,
and $(\uu,\vp)$ is the weak solution departing from 
$(\uu_0,\vp_0)$.
Then, there exists two positive constants $M_4=M_4(R,m,\tau)$ and $M_5=M_5(R,m,\tau)$, independent of the specific datum $(\uu_0,\vp_0)$, such that
\begin{equation}
\label{reg4}
\| \partial_t \uu\|_{L^{\infty}(\tau,\infty;\H_\sigma)}
+\| \partial_t \vp\|_{L^{\infty}(\tau,\infty;H)}
\leq M_4,
\end{equation}
and 
\begin{equation}
\label{reg5}
\| \partial_t \uu\|_{L^2(t,t+1;\V_\sigma)}
+\|\partial_t \vp\|_{L^2(t,t+1;H^2(\Omega))} \leq M_5,\quad \forall \, t \geq \tau.
\end{equation}
Furthermore, there exists $\delta=\delta(R,m,\tau)>0$ and $M_6=M_6(R,m,\tau)$ such that
$$
\sup_{ t\geq \tau}\|\varphi(t)\|_{\mathcal{C}(\overline{\Omega} )}\leq 1-\delta
$$
and 
\begin{equation}
\label{reg6}
\sup_{t\geq \tau}\| \uu(t)\|_{\W_\sigma}+ \sup_{t\geq \tau}\| \vp(t)\|_{H^4(\Omega)}\leq M_6.
\end{equation}
\end{theorem}
\medskip

\begin{proof}
First, by replacing $\tau$ with $\frac{\tau}{2}$ in Theorem \ref{regularity}, we can assume that the solution $(\uu,\vp)$
satisfies the uniform estimates \eqref{reg1}-\eqref{reg3} on the time interval $[\frac{\tau}{2},\infty)$.
We proceed by showing additional higher-order a priori estimates on the solution. In the sequel, $k_i$, $i\in\mathbb{N}$, denotes a positive constant which depends on $R$, $m$ and $\tau$, but is independent of the specific initial datum.
Given $h>0$, repeating line by line the proof of the stability result \eqref{uniqhigh} in Theorem \ref{strong} (cf. Step $8$), 
where the difference of two solutions $(\uu_1-\uu_2, \vp_1-\vp_2)$ is replaced 
by $(\partial_t^h \uu, \partial_t^h \vp)$, we deduce the differential inequality 
\begin{equation}
\label{diffhigh}
\ddt \mathcal{H}_2
+ \frac{\nu_\ast}{2} \| \nabla \partial_t^h \uu\|^2 
+\frac18\|\Delta \partial_t^h\varphi\|^2
\leq k_0 (1+ \| \uu\|_{\mathbf{W}^{1,3}(\Omega)}^4) \mathcal{H}_2,
\end{equation}
where
$$
\mathcal{H}_2= \frac12 \| \partial_t^h \uu\|^2
+ \frac12 \|\partial_t^h\varphi\|^2,
$$
and the positive constant $k_0$ is independent of $h$, but depends on $M_1$ and $M_3$.
Recalling that
$\| \partial_t^h f\|_{L^2(t,t+1; H)}
\leq\| f_t\|_{L^2(t,t+2; H)},$
thanks to Theorem \ref{regularity}, we observe that
$$
\int_t^{t+1} \Big( \mathcal{H}_2(s)+ \| \uu(s)\|_{\mathbf{W}^{1,3}(\Omega)}^4\Big) \, \d s\leq k_1, \quad \forall \, t \geq \frac{\tau}{2},
$$
where $k_1$ is independent of $h$, but depends on $M_2$.
Hence, the uniform Gronwall lemma (see \cite[Chapter III, Lemma 1.1]{T}) with $r=\frac{\tau}{2}$ yields
$$
\| \partial_t^h \uu\|_{L^{\infty}(\tau,\infty;\H_\sigma)}
+\| \partial_t^h \vp\|_{L^{\infty}(\tau,\infty;H)}
\leq M_4,
$$
and 
$$
\|\partial_t^h \uu\|_{L^2(t,t+1;\V_\sigma)}
+\| \partial_t^h \vp\|_{L^2(t,t+1;H^2(\Omega))} \leq M_5, \quad \forall \, t \geq \tau,
$$
where $M_4$ and $M_5$ depend on $R$, $m$ and $\tau$, but are independent of $h$, $t$ and the specific initial datum. 
A final passage to the limit as $h\rightarrow 0$ entails \eqref{reg4} 
and \eqref{reg5}.
We are now in position to prove the separation property.
In light of \eqref{reg4}, it is immediate to deduce that  
$\partial_t\vp+\uu\cdot \nabla \vp\in L^\infty(\tau,\infty;H)$. Then, the regularity theory of the Neumann problem implies that 
\begin{equation}
\label{muH2}
\| \mu\|_{L^{\infty}(\tau,\infty;H^2(\Omega))}\leq k_2.
\end{equation}
By using Theorem \ref{ell2} with $f=\mu+\theta_0\vp\in L^\infty(\Omega \times(\tau,\infty))$, we find
$
\| F'(\vp)\|_{L^\infty(\Omega\times(\tau,\infty))}
\leq k_3.
$
This, in turn, entails that there exists $\delta>0$ such that
\begin{equation}
\label{separation}
\sup_{t\geq \tau}\| \vp(t)\|_{\mathcal{C}
(\overline{\Omega})} \leq 1-\delta.
\end{equation}
Thanks to the regularity \eqref{reg3} and the separation property \eqref{separation}, and recalling that $F\in \mathcal{C}^3([-1+\delta,1-\delta])$, we deduce that 
$\| F'(\vp)\|_{L^\infty(\tau,\infty;H^2(\Omega))}\leq k_4.$
Thus, exploiting \eqref{muH2}, the above control and the regularity theory of the Neumann problem, we get
$\|\vp\|_{L^{\infty}(\tau,\infty;H^4(\Omega))} \leq k_5.$
Moreover, setting $\f= \mu \nabla \vp-\partial_t\uu -(\uu\cdot \nabla )\uu$,
we infer from  \eqref{reg1}, \eqref{reg3} and \eqref{reg4} that, for any $1<p<2$, there exists $k_6$ such that
$
\| \f\|_{L^\infty(\tau,\infty;\mathbf{L}^p(\Omega))}\leq k_6,
$
where $k_6$ depends on $p$. Then, in light of \eqref{reg3}, 
an application of Theorem \ref{STOKESVISCREG} (with $r=\infty$) yields
$
\| \uu\|_{L^\infty(\tau,\infty; \mathbf{W}^{2,p}(\Omega))}\leq k_7.
$
Recalling the embedding $W^{1,p}\hookrightarrow L^{p^\ast}$, where 
$\frac{1}{p^\ast}=\frac{1}{p}-\frac{1}{2}$, and choosing $p=\frac43$, we obtain  
$\uu \in L^\infty(\tau,\infty; \mathbf{W}^{1,4}(\Omega))$.
Thanks to this regularity, we observe that $\f\in L^\infty(\tau,\infty;\H)$. 
Applying once again Theorem \ref{STOKESVISCREG}, we find 
$$
\| \uu\|_{L^{\infty}(\tau, \infty;\W_\sigma)}\leq k_8.
$$
Due to the continuity in time of the solution, 
we note that the above inequalities hold for any $t\geq \tau$, giving the desired estimate \eqref{reg6} with $M_6$ depending on $k_5$ and $k_8$.
\end{proof}

\section{Local Strong Solutions in Three Dimensions}
\label{6}
\setcounter{equation}{0}

\noindent
In this section we study the well-posedness of strong solutions in dimension three. 
\begin{theorem}
\label{strong2}
Let $d=3$.
Assume that $\uu _0\in \V_\sigma$ and $\vp_0 \in H^2(\Omega)$ is such that 
$\| \vp_0\|_{L^\infty(\Omega)}\leq 1$, $|\overline{\vp}_0|<1$, $\mu_0=-\Delta \vp_0+\Psi'(\vp_0) \in V$ and 
$\partial_\n \vp_0=0$ on $\partial \Omega$.
Then, there exist a time $T^\ast>0$ and a unique strong solution to \eqref{system}-\eqref{bcic} on $[0,T^\ast]$ satisfying
\begin{align*}
&\uu \in L^\infty(0,T^\ast;\V_\sigma)\cap L^2(0,T^\ast;\W_\sigma)
\cap H^1(0,T^\ast;\H_\sigma),\quad \pi \in L^2(0,T^\ast;V),\\
&\vp \in L^\infty(0,T^\ast;W^{2,6}(\Omega))\cap H^1(0,T^\ast;V),\\
&\mu \in L^{\infty}(0,T^\ast;V)\cap L^2(0,T^\ast;H^3(\Omega)).
\end{align*}
The strong solution satisfies \eqref{system} almost everywhere on 
$(x,t)\in \Omega \times (0,T^\ast)$ and $\partial_\n \mu=0$ almost everywhere on 
$\partial\Omega\times(0,T^\ast)$.
\end{theorem}

\noindent
The proof of Theorem \ref{strong2}  relies on the argument employed in the proofs of Theorems \ref{weak-uniqueness} and \ref{strong}. For the sake of brevity, we report only the main changes.

\begin{proof}
We follow the proof of Theorem \ref{strong}. For the same values of $k$, $\varepsilon$ and $n$ as defined in Steps $1$-$3$, we obtain the approximating sequences 
$(\uu_{k,\varepsilon}^n, \vp_{k,\varepsilon}^n)$ which solve
\eqref{e1app}-\eqref{e2app} and \eqref{muapp}.
Before deriving uniform a priori estimates we specify that the positive constant $c'_i$, $i \mathbb{N}$, depends on the parameters of the system, the constants in embedding and interpolation results, and $\E(\uu_0,\vp_0)$, but is independent of the approximation parameters $k$, $\varepsilon$ and $n$ and of the norms $\| \uu_0\|_{\V_\sigma}$ and $\| \mu_0\|_V$.
It is easily seen that the energy estimates \eqref{dissest1app}-\eqref{vptappr} also hold. In particular, we have
\begin{equation}
\label{dissest1app3D}
\| \uu_{k,\varepsilon}^n(t)\|+ \| \vp_{k,\varepsilon}^n(t)\|_V\leq c'_1, \quad \forall \, t \in [0,T].
\end{equation}
Let us now proceed by showing higher-order Sobolev estimates.
First, arguing as in Step $5$ we find 
\begin{align}
\label{diffineq1app3D}
\ddt \Big[ &\frac12 \| \nabla \mu_{k,\varepsilon}^n\|^2
+ ( \uu_{k,\varepsilon}^n\cdot 
\nabla \vp_{k,\varepsilon}^n,
\mu_{k,\varepsilon}^n)
\Big]
+ \frac14 \| \nabla 
\partial_t \vp_{k,\varepsilon}^n\|^2 \notag\\
&\leq (\partial_t \uu_{k,\varepsilon}^n \cdot \nabla \vp_{k,\varepsilon}^n, \mu_{k,\varepsilon}^n)+ 
c'_2 (1+\| \uu_{k,\varepsilon}^n\|_{\mathbf{L}^3(\Omega)}^2)
(1+ \| \nabla \uu_{k,\varepsilon}^n\|^2+\|\nabla \mu_{k,\varepsilon}^n\|^2).
\end{align}
In order to recover estimates on the velocity field, we take first
$\vv =\partial_t \uu_{k,\varepsilon}^n$ in \eqref{e1app}. This yields
$$
\| \partial_t \uu_{k,\varepsilon}^n\|^2+
b(\uu_{k,\varepsilon}^n, \uu_{k,\varepsilon}^n,\partial_t \uu_{k,\varepsilon}^n)
-(\mathrm{div }\, (\nu(\vp_{k,\varepsilon}^n)D\uu_{k,\varepsilon}^n), \partial_t \uu_{k,\varepsilon}^n)
= (\mu_{k,\varepsilon}^n \nabla \vp_{k,\varepsilon}^n, \partial_t \uu_{k,\varepsilon}^n).
$$
By using \eqref{LADY3}, \eqref{H2equiv}, we have
\begin{align*}
b(\uu_{k,\varepsilon}^n, \uu_{k,\varepsilon}^n,
\partial_t \uu_{k,\varepsilon}^n)&\leq 
\| \uu_{k,\varepsilon}^n\|_{\mathbf{L}^6(\Omega)}
 \| \nabla  \uu_{k,\varepsilon}^n\|_{\mathbf{L}^3(\Omega)} 
\|\partial_t \uu_{k,\varepsilon}^n \|\\
&\leq C\| \nabla  \uu_{k,\varepsilon}^n \|^\frac32 
\| \A \uu_{k,\varepsilon}^n\|^{\frac12} 
\|\partial_t \uu_{k,\varepsilon}^n \|\\
&\leq \frac16  \|\partial_t \uu_{k,\varepsilon}^n \|^2+ 
c'_3 \Big( \| \A \uu_{k,\varepsilon}^n\|^2+\| \nabla  \uu_{k,\varepsilon}^n \|^6\Big).
\end{align*}
Exploiting once more \eqref{LADY3} and \eqref{H2equiv}, we obtain
\begin{align*}
(\mathrm{div }\, (\nu(\vp_{k,\varepsilon}^n)
D\uu_{k,\varepsilon}^n), 
\partial_t \uu_{k,\varepsilon}^n)
&=\frac12( \nu(\vp_{k,\varepsilon}^n) \Delta \uu_{k,\varepsilon}^n,  \partial_t \uu_{k,\varepsilon}^n)+
(\nu'(\vp_{k,\varepsilon}^n)
  D \uu_{k,\varepsilon}^n \nabla \vp_{k,\varepsilon}^n, 
\partial_t \uu_{k,\varepsilon}^n)\\
&\leq  C \| \A \uu_{k,\varepsilon}^n\| \|\partial_t \uu_{k,\varepsilon}^n \|
+ C \| \nabla \vp_{k,\varepsilon}^n\|_{\mathbf{L}^6(\Omega)}
\| D\uu_{k,\varepsilon}^n\|_{\mathbf{L}^3(\Omega)} 
\|\partial_t \uu_{k,\varepsilon}^n \|\\
&\leq \frac16  \|\partial_t \uu_{k,\varepsilon}^n \|^2+
C \| \A \uu_{k,\varepsilon}^n\|^2
+ C \| \vp_{k,\varepsilon}^n\|_{H^2(\Omega)}^2
\| \nabla \uu_{k,\varepsilon}^n\| 
\|\A \uu_{k,\varepsilon}^n \|\\
&\leq \frac16  \|\partial_t \uu_{k,\varepsilon}^n \|^2+
c'_4 \Big( \| \A \uu_{k,\varepsilon}^n\|^2
+  \| \vp_{k,\varepsilon}^n\|_{H^2(\Omega)}^4
\| \nabla \uu_{k,\varepsilon}^n\|^2 \Big).
\end{align*}
On the other hand, by \eqref{dissest3app} we have
\begin{align*}
(\mu_{k,\varepsilon}^n\nabla \vp_{k,\varepsilon}^n, 
\partial_t \uu_{k,\varepsilon}^n)
&\leq \| \mu_{k,\varepsilon}^n\|_{L^6(\Omega)}
\| \nabla \vp_{k,\varepsilon}^n\|_{\mathbf{L}^3(\Omega)}
\| \partial_t \uu_{k,\varepsilon}^n\| \\
&\leq  \frac16  \|\partial_t \uu_{k,\varepsilon}^n \|^2+
c'_5 \| \nabla \vp_{k,\varepsilon}^n\|_{\mathbf{L}^3(\Omega)}^2 
(1+\| \nabla \mu_{k,\varepsilon}^n\|^2).
\end{align*}
Collecting the above estimates, we arrive at 
\begin{align}
\label{ut3D}
\| \partial_t \uu_{k,\varepsilon}^n\|^2
&\leq c'_6 \Big( \| \A \uu_{k,\varepsilon}^n\|^2+
\| \nabla  \uu_{k,\varepsilon}^n \|^6 \notag \\
&\quad +  \| \vp_{k,\varepsilon}^n\|_{H^2(\Omega)}^4
\| \nabla \uu_{k,\varepsilon}^n\|^2
+ \| \nabla \vp_{k,\varepsilon}^n\|_{\mathbf{L}^3(\Omega)}^2 
(1+\| \nabla \mu_{k,\varepsilon}^n\|^2) \Big).
\end{align}
Next, we take $\vv=\A \uu_{k,\varepsilon}^n$ in \eqref{e1app}. We recall that there exists $p_{k,\varepsilon}^n\in L^2(0,T;V)$ satisfying 
$-\Delta \uu_{k,\varepsilon}^n+\nabla p_{k,\varepsilon}^n
= \A \uu_{k,\varepsilon}^n$ almost everywhere in $\Omega \times (0,T)$ and the estimates \eqref{pressure2}. 
Thus, we find
\begin{align*}
\frac12 \ddt \| \nabla \uu_{k,\varepsilon}^n\|^2+
\nu_\ast \| \A \uu_{k,\varepsilon}^n\|^2&\leq 
-b(\uu_{k,\varepsilon}^n, \uu_{k,\varepsilon}^n,
\A \uu_{k,\varepsilon}^n)
-(\nu'(\vp_{k,\varepsilon}^n)
\nabla \vp_{k,\varepsilon}^n p_{k,\varepsilon}^n,\A \uu_{k,\varepsilon}^n) \\
&\quad + (\nu'(\vp_{k,\varepsilon}^n)D \uu_{k,\varepsilon}^n \nabla \vp_{k,\varepsilon}^n,
\A \uu_{k,\varepsilon}^n)
+(\mu_{k,\varepsilon}^n \nabla \vp_{k,\varepsilon}^n,
 \A \uu_{k,\varepsilon}^n).
\end{align*}
We address the right-hand side of the above differential inequality by using \eqref{LADY3} and \eqref{pressure2}. We have
\begin{align*}
-(\nu'(\vp_{k,\varepsilon}^n) \nabla 
&\vp_{k,\varepsilon}^n p_{k,\varepsilon}^n,\A \uu_{k,\varepsilon}^n)
 +(\nu'(\vp_{k,\varepsilon}^n)  D \uu_{k,\varepsilon}^n \nabla \vp_{k,\varepsilon}^n, 
\A \uu_{k,\varepsilon}^n)\\
&\leq C \| \nabla \vp_{k,\varepsilon}^n\|_{\mathbf{L}^6(\Omega)} 
\Big( \| p_{k,\varepsilon}^n\|_{L^3(\Omega)}+ 
\| D \uu_{k,\varepsilon}^n\|_{\mathbf{L}^3(\Omega)} \Big)
\|\A \uu_{k,\varepsilon}^n \| \\
&\leq C \| \vp_{k,\varepsilon}^n\|_{H^2(\Omega)}
\Big( \| p_{k,\varepsilon}^n\|^\frac12 
\| p_{k,\varepsilon}^n\|_V^\frac12 +
\| \nabla \uu_{k,\varepsilon}^n\|^\frac12 \| \A \uu_{k,\varepsilon}^n\|^\frac12 \Big)
\|\A \uu_{k,\varepsilon}^n \| \\
&\leq C \| \vp_{k,\varepsilon}^n\|_{H^2(\Omega)}
\Big( \|\nabla \uu_{k,\varepsilon}^n \|^\frac14 
\|\A \uu_{k,\varepsilon}^n\|^\frac34+ \| \nabla \uu_{k,\varepsilon}^n\|^\frac12 \| \A \uu_{k,\varepsilon}^n\|^\frac12 \Big) \| \A \uu_{k,\varepsilon}^n\|\\
&\leq \frac{\nu_\ast}{6} \|\A \uu_{k,\varepsilon}^n \|^2 +
c'_7 (1+\| \vp_{k,\varepsilon}^n\|_{H^2(\Omega)}^8)
\| \nabla \uu_{k,\varepsilon}^n\|^2,
\end{align*}
and
\begin{align*}
b(\uu_{k,\varepsilon}^n, \uu_{k,\varepsilon}^n,\A \uu_{k,\varepsilon}^n)
&\leq \| \uu_{k,\varepsilon}^n\|_{\mathbf{L}^6(\Omega)}
\| \nabla \uu_{k,\varepsilon}^n\|_{\mathbf{L}^3(\Omega)}
\| \A \uu_{k,\varepsilon}^n\| \\
&\leq \frac{\nu_\ast}{6} \|\A \uu_{k,\varepsilon}^n\|^2 + 
c'_8 \| \nabla  \uu_{k,\varepsilon}^n\|^6.
\end{align*}
Moreover, we have
\begin{align*}
(\mu_{k,\varepsilon}^n \nabla \vp_{k,\varepsilon}^n, \A \uu_{k,\varepsilon}^n)
&\leq \|\mu_{k,\varepsilon}^n\|_{L^6(\Omega)} 
\|\nabla \vp_{k,\varepsilon}^n\|_{\mathbf{L}^3(\Omega)} 
\| \A \uu_{k,\varepsilon}^n\|\\
&\leq  \frac{\nu_\ast}{6} \| \A \uu_{k,\varepsilon}^n\|^2
+c'_9 \|\nabla \vp_{k,\varepsilon}^n\|_{\mathbf{L}^3(\Omega)}^2
 (1+\| \nabla \mu_{k,\varepsilon}^n\|^2).
\end{align*}
Combining these estimates, we obtain
\begin{align}
\label{diffineq2app3D}
\frac12 &\ddt \| \nabla \uu_{k,\varepsilon}^n\|^2+ 
\frac{\nu_\ast}{2}  \| \A \uu_{k,\varepsilon}^n\|^2\notag \\
&\leq c'_{10}\Big(  (1+\| \vp_{k,\varepsilon}^n\|_{H^2(\Omega)}^8)
\| \nabla \uu_{k,\varepsilon}^n\|^2 
+ \| \nabla  \uu_{k,\varepsilon}^n\|^6
+\|\nabla \vp_{k,\varepsilon}^n\|_{\mathbf{L}^3(\Omega)}^2
 (1+\| \nabla \mu_{k,\varepsilon}^n\|^2) \Big).
\end{align}
Multiplying \eqref{ut3D} by $\varpi= \frac{\nu_\ast}{4c'_6}>0$ 
and summing up to \eqref{diffineq2app3D}, we obtain
\begin{align}
\label{diffineq2app23D}
\frac12 &\ddt \| \nabla \uu_{k,\varepsilon}^n\|^2+ 
\frac{\nu_\ast}{4}  \| \A \uu_{k,\varepsilon}^n\|^2+
\varpi \| \partial_t \uu_{k,\varepsilon}^n\|^2\notag \\
&\leq c'_{11} \Big(  (1+\| \vp_{k,\varepsilon}^n\|_{H^2(\Omega)}^8)
\| \nabla \uu_{k,\varepsilon}^n\|^2  
+ \| \nabla  \uu_{k,\varepsilon}^n\|^6
+\|\nabla \vp_{k,\varepsilon}^n\|_{\mathbf{L}^3(\Omega)}^2
 (1+\| \nabla \mu_{k,\varepsilon}^n\|^2) \Big).
\end{align}
Adding \eqref{diffineq1app3D} to \eqref{diffineq2app23D}, we find the differential inequality
\begin{align}
\label{diffineq3app3D}
\ddt \Lambda(\uu_{k,\varepsilon}^n,
\vp_{k,\varepsilon}^n)&
+ \frac{\nu_\ast}{8} \|\A \uu_{k,\varepsilon}^n\|^2
+\frac{\varpi}{2} \|\partial_t \uu_{k,\varepsilon}^n\|^2
+ \frac14 \| \nabla \partial_t \vp_{k,\varepsilon}^n\|^2 \notag\\
&\leq  (\partial_t \uu_{k,\varepsilon}^n \cdot \nabla \vp_{k,\varepsilon}^n, \mu_{k,\varepsilon}^n)+ c'_{12} \Big( (1+\| \vp_{k,\varepsilon}^n\|_{H^2(\Omega)}^8)
\| \nabla \uu_{k,\varepsilon}^n\|^2  
+ \| \nabla \uu_{k,\varepsilon}^n\|^6 \notag\\
&\quad  
+ ( 1+\| \nabla \vp_{k,\varepsilon}^n\|_{\mathbf{L}^3(\Omega)}^2 
+ \| \uu_{k,\varepsilon}^n\|_{\mathbf{L}^3(\Omega)}^2) 
( 1+\| \nabla \uu_{k,\varepsilon}^n\|^2+ \| \nabla \mu_{k,\varepsilon}^n\|^2) \Big),
\end{align}
where $\Lambda(\uu_{k,\varepsilon}^n, \vp_{k,\varepsilon}^n)$ is the same as in \eqref{Lambdadef}.
Owing to \eqref{LADY3} and \eqref{dissest1app3D}, we observe that
\begin{align*}
( \uu_{k,\varepsilon}^n\cdot 
\nabla \vp_{k,\varepsilon}^n,\mu_{k,\varepsilon}^n) 
&\leq \|\uu_{k,\varepsilon}^n\|_{\mathbf{L}^3(\Omega)} 
\| \nabla \vp_{k,\varepsilon}^n\| 
\| \mu_{k,\varepsilon}^n\|_{L^6(\Omega)}\\
&\leq \frac14 \| \nabla \uu_{k,\varepsilon}^n\|^2+ \frac14\| \nabla \mu_{k,\varepsilon}^n\|^2+c'_{13}.
\end{align*}
Thus, we deduce that
\begin{equation}
\label{lambdabelow3D}
\Lambda(\uu_{k,\varepsilon}^n, \vp_{k,\varepsilon}^n) \geq \frac14 \| \nabla \uu_{k,\varepsilon}^n\|^2
+ \frac14 \| \nabla \mu_{k,\varepsilon}^n\|^2-c'_{13}.
\end{equation}
On the other hand, we have
$$
\Lambda(\uu_{k,\varepsilon}^n, \vp_{k,\varepsilon}^n) 
\leq C \| \nabla \uu_{k,\varepsilon}^n\|^2
+ C\| \nabla \mu_{k,\varepsilon}^n\|^2+c'_{14}.
$$
Exploiting \eqref{vpapprH2}, we are led to
\begin{align}
\label{diffineq43D}
\ddt \Lambda(\uu_{k,\varepsilon}^n, \vp_{k,\varepsilon}^n) 
&+ \overline{\nu} \Big( \|\A \uu_{k,\varepsilon}^n\|^2
+ \| \partial_t \uu_{k,\varepsilon}^n\|^2
+  \| \nabla \partial_t \vp_{k,\varepsilon}^n\|^2 \Big)
\leq  c'_{15} \Big( 1+\Lambda^3(\uu_{k,\varepsilon}^n, \vp_{k,\varepsilon}^n)\Big),
\end{align}
where $\overline{\nu}=\frac14  \min \lbrace 1, \nu_\ast, \varpi\rbrace$.
In addition, following line by line the estimates performed  
in the proof of Theorem \ref{strong} for a uniform bound of
the initial condition, we easily get 
\begin{equation}
\label{lambda0-2}
\Lambda(\uu_\la^n(0), \vp_\la^n(0))\leq C(1+\| \uu_0\|_{\V_\sigma}+\| \mu_0\|_V),
\end{equation}
where $C$ is independent of $k$, $\varepsilon$ and $n$, provided that $n$ is sufficiently large.
Therefore, we infer from \eqref{diffineq43D} and \eqref{lambda0-2} that there exist a positive time $T^\ast$, depending on $\| \uu_0\|_{\V_\sigma}$ and $\| \mu_0\|_V$, and a positive constant $\overline{C}$ (independent of $k$, $\varepsilon$ and $n$) 
such that
$$
\sup_{0\leq t\leq T^\ast} \Lambda(\uu_{k,\varepsilon}^n(t), 
\vp_{k,\varepsilon}^n(t))+ \int_0^{T^\ast} 
\Big( \|\A \uu_{k,\varepsilon}^n(s)\|^2
+\| \partial_t \uu_{k,\varepsilon}^n(s)\|^2
+  \| \nabla \partial_t \vp_{k,\varepsilon}^n(s)\|^2 \Big)\, \d s \leq \overline{C}.
$$
A final passage to the limit allows us to recover the existence of a 
strong solution to the original problem \eqref{system}-\eqref{bcic}.
Moreover, the additional claimed regularities for $\vp$ and $\mu$ can be easily deduced as in the proof of Theorem \ref{strong}.

We are left to prove the uniqueness of strong solutions. 
Given two strong solutions $(\uu_1,\vp_1)$ and $(\uu_2,\vp_2)$, defined on the time interval $(0,T_0)$ with the same initial datum $(\uu_0,\vp_0)$, we define their difference 
$\uu=\uu_1-\uu_2$ and $\vp=\vp_1-\vp_2$. We observe that
the regularity of strong solutions allows us to follow the argument in the proof of Theorem \ref{weak-uniqueness}. Then, we have the differential inequality
\begin{align}
\label{inequniq3D}
\ddt \mathcal{H} +\nu_\ast \| \uu\|^2+ \frac12 \|\nabla \vp\|^2\leq \alpha^2 \mathcal{H}+ \sum_{k=1}^7 \mathcal{I}_k,
\end{align}
where the terms $\mathcal{H}$ and $\mathcal{I}_k$ are defined as above. 
In light of the regularity $\uu_i \in L^\infty(0,T_0;\V_\sigma)$ and $\vp_i \in L^\infty(0,T_0;W^{2,6}(\Omega))$, $i=1,2$, 
we can easily infer that
\begin{align*}
\mathcal{I}_1+\mathcal{I}_2+\mathcal{I}_5+\mathcal{I}_6+\mathcal{I}_7
\leq\frac{1}{6}\| \nabla \vp\|^2+\frac{\nu_\ast}{8} \| \uu\|^2 
+ C_1 \Big( \| \vp\|_{\ast}^2+ \| \uu\|_{\sharp}^2\Big),
\end{align*}
for some positive constant $C_1$.
On the other hand, by using \eqref{LADY3} and the boundedness of $\nu'$, we simply obtain
\begin{align*}
\mathcal{I}_3&\leq C \| \vp\|_{L^6(\Omega)} \| D\uu_2\|_{\mathbf{L}^3(\Omega)} 
\| \nabla \A^{-1} \uu\|\\
&\leq \frac{1}{12} \| \nabla \vp\|^2 +C_2 \| D \uu_2\|_{\mathbf{L}^3}^2 \| \uu\|_{\sharp}^2,
\end{align*}
and
\begin{align*}
\mathcal{I}_4&\leq \Big( \| \uu_1\|_{\mathbf{L}^6(\Omega)} 
+ \| \uu_2\|_{\mathbf{L}^6(\Omega)}\Big) \| \uu\| 
\| \nabla \A^{-1} \uu\|_{\mathbf{L}^3(\Omega)} \\
&\leq C \Big( \| \uu_1\|_{\mathbf{L}^6(\Omega)} 
+ \| \uu_2\|_{\mathbf{L}^6(\Omega)}\Big) \| \uu\|^{\frac32} \| \nabla \A^{-1}\uu\|^{\frac12}\\
&\leq \frac{\nu_\ast}{8} \| \uu\|^2+ C_3\| \uu\|_{\sharp}^2,
\end{align*}
for some positive constants $C_2$ and $C_3$.
Collecting the above estimates together, we end up with
$$
\ddt \mathcal{H}\leq C_4(1+\| D\uu_2\|_{L^3(\Omega)}^2)\mathcal{H}.
$$
Since $D\uu_2 \in L^2(0,T_0;\mathbf{L}^3(\Omega))$, the uniqueness of strong solutions immediately follows from the Gronwall lemma. 
\end{proof}

\appendix
\section{On Neumann Problems}
\label{Neumann-Laplace}
\setcounter{equation}{0}

\noindent
For any $\lambda\geq 0$, let us consider the Neumann problem 
\begin{equation}
\label{Neumannprob}
\begin{cases}
-\Delta u+\lambda u=f,\quad &\text{ in }\Omega,\\
\partial_\n u=0, \quad &\text{ on }\partial \Omega.
\end{cases}
\end{equation}
We introduce the operator $B_\lambda \in \mathcal{L}(V,V^{\prime })$ defined by
\begin{equation*}
\l B_\lambda u,v\r  =\int_{\Omega } ( \nabla u\cdot
\nabla v + \lambda u v ) \, \d x,\quad \forall \, u,v\in V.
\end{equation*}%
We consider the spaces
\begin{equation*}
V_{0}=\left\{ v\in V:\overline{v}=0\right\} ,\quad V_{0}^{\prime }
=\left\{ f\in V^{\prime
}:\overline{f}=0\right\},
\end{equation*}%
and we recall that $V=V_0\oplus \mathbb{R}$ and $V'= V_0'\oplus \mathbb{R}$.
The restriction $A_0$ of $B_0$ to $V_{0}$ being an isomorphism from $V_{0}$ onto $%
V_{0}^{\prime }$, we denote by $A_0^{-1}:V_{0}^{\prime
}\rightarrow V_{0}$ its inverse map. 
It is well known that for all $f\in V_{0}^{\prime }$, $A_0^{-1}f$ is the
unique $u\in V_{0}$ such that $\l A_0 u,v\r= \l f,v \r$, for all $v \in V$.
On account of the above definitions, we observe that
\begin{align}
\l  f, A_0^{-1}g\r &
=\int_{\Omega }\nabla (A_0^{-1}
f)\cdot \nabla (A_0^{-1}g)\, \d x,\quad \forall \, f,g\in V_{0}^{\prime }.
\label{propN2}
\end{align}%
Owing to \eqref{propN2}, it is straightforward to prove that
$
\Vert f\Vert _{\ast }:=\Vert \nabla A_0^{-1}f\Vert =\l f,
A_0^{-1}f\r^{\frac12}
$
is a norm on $V_{0}^{\prime }$ equivalent to the natural one.
In addition, for any $u \in H^{1}(0,T;V_{0}^{\prime })$, we have the chain rule
\begin{equation}
\l u_{t}\left(t\right),A_0^{-1}u\left( t\right) \r =
\frac12 \ddt \Vert u(t)\Vert _{\ast }^{2},\quad
\text{ a.e. }t\in (0,T).
\label{propN4}
\end{equation}%
Furthermore, due to regularity theory of the Neumann problem, we know that 
\begin{equation}
\label{neumannest}
\| \nabla A_0^{-1} f\|_{V}\leq C \| f\|,\quad \forall \, f\in H\cap V_0'.
\end{equation}
For any $\lambda>0$, we also consider the operator $A_\lambda= -\Delta +\lambda I$ as unbounded operator on $H$ with domain $D(A_\lambda)= \lbrace u\in H^2(\Omega): \partial_\n u=0 \text{ on } \partial \Omega \rbrace$. It is well-known that $A_\lambda$ is positive, unbounded, self-adjoint operator in $H$ with compact inverse (see, e.g., \cite[Chapter II, Section 2.2]{T}).
\smallskip

Next, we introduce the homogeneous Neumann elliptic problem with a logarithmic convex nonlinear term, that is, with the same $F$ as in \eqref{Singpot-form}-\eqref{H},
\begin{equation}
\label{ELL}
\begin{cases}
-\Delta u+F'(u)=f,\quad &\text{ in }\Omega,\\
\partial_\n u=0, \quad &\text{ on }\partial \Omega.
\end{cases}
\end{equation}

\noindent
Under the assumptions in Section \ref{2}, we have the 
following well-posedness and approximation result.

\begin{lemma}
\label{existenceNP}
Let $\Omega$ be a bounded domain in $\mathbb{R}^d$, $d=2,3$, with smooth boundary. 
Assume that $f\in H$. Then, there exists a unique solution $u$ to problem 
\eqref{ELL} such that $u\in H^2(\Omega)$, $F'(u)\in H$ and satisfies 
$-\Delta u+F'(u)=f$ for almost every $x\in \Omega$ and $\partial_{\n} u=0$ for almost every $x\in \partial \Omega$. Moreover, we have
\begin{equation}
\label{estbase}
\| u\|_{H^2(\Omega)}+\| F'(u)\|\leq C \big( 1+ \| f\|\big).
\end{equation}
Let us assume that the sequence 
$\left\lbrace f_k\right\rbrace \subset H$, and $f\in H$.
We consider the solutions $u_k$ and $u$ to problem \eqref{ELL} corresponding to $f_k$ and $f$, respectively. Then, $f_k\rightarrow f$ in  $H$, as $k\rightarrow \infty$, implies 
\begin{equation}
\label{approximationresult}
\| u_k-u\|_V\rightarrow 0,\text{ as } k\rightarrow \infty.
\end{equation}
\end{lemma}

\begin{proof}
The existence of a solution $u$ to problem \eqref{ELL} can be proved relying on the theory of maximal monotone operator. We define the functional on $H$
$$
\mathcal{F}(u)= \int_{\Omega} \frac12 \| \nabla u\|^2 + F(u)\, \d x,
$$
with domain $D(\mathcal{F})=\lbrace u\in H^1(\Omega): \| u\|_{L^{\infty}(\Omega)}\leq 1\rbrace$. We observe that $\mathcal{F}$ is a proper, lower semi-continuous and convex functional. 
Now, we consider the subdifferential $\partial \mathcal{F}$ of $\mathcal{F}$, defined as $w\in \partial \mathcal{F}(u)$ if and only if, for all $v\in H$, $\mathcal{F}(v)\geq \mathcal{F}(u)+ (w,v-u)$.
Then, $\partial \mathcal{F}$ is a maximal monotone operator on $H$ (see \cite{BREZIS}). Moreover, it is well-known that 
$D(\partial \mathcal{F})= \lbrace u\in H^2(\Omega): F'(u)\in H, \partial_{\n} u= 0 \text{ on } \partial \Omega \rbrace$ and 
$\partial \mathcal{F}(u)= -\Delta u + F'(u)$ (see \cite{BARBU,AW}).
By \eqref{H}, we deduce that $\partial \mathcal{F}$ is also coercive, namely $(\partial \mathcal{F}(u)-\partial \mathcal{F}(v), u-v)\geq \theta \| u-v\|^2$, for all $u,v \in D(\partial \mathcal{F})$, where $\theta$ is the same as in \eqref{H}. In turn, this implies that $\partial \mathcal{F}$ is surjective on $H$. In addition, the estimate \eqref{estbase} can be proved as in \cite{AW,CG}. 
Finally, exploiting \eqref{H} once more, we can easily infer the uniqueness of solutions and the approximation result \eqref{approximationresult} to problem \eqref{ELL}.   
\end{proof}

We now report some elliptic estimates, whose proofs can be found in \cite{A,CG,GGM}.
\begin{theorem}
\label{ell2}
Let $\Omega$ be a bounded domain in $\mathbb{R}^d$ with smooth boundary. 
Assume that $u$ is the solution to problem \eqref{ELL}. We have the following:
\smallskip

\noindent
$\diamond$ Let $d=2,3$ and $f\in L^p(\Omega)$, where $2\leq p\leq \infty$.
Then, we have $$
\|F'(u)\|_{L^p(\Omega)}\leq \| f\|_{L^p(\Omega)}.
$$
\item[$\diamond$] Let $d=2,3$ and $f\in V$.
Then, we have
$$
\| \Delta u\|\leq \| \nabla u\|^{\frac12} \| \nabla f\|^{\frac12}.
$$
In addition, there exists a positive constant $C=C(p)$ such that
$$
\|u \|_{W^{2,p}(\Omega)}+ \|F'(u)\|_{L^p(\Omega)}
\leq C \big( 1+\| f\|_{V}\big),
$$
where $p=6$ if $d=3$ and for any $p\geq 2$ if $d=2$.
\item[$\diamond$] Let $d=2$ and $f \in V$. Assume that $F$
satisfies
$$
F''(s)\leq e^{C|F'(s)|+C}, \quad \forall \, s \in (-1,1),
$$
for some positive constant $C$.
Then, for any $p\geq 1$, there exists a positive constant $C=C(p)$ such that
$$
\| F''(u)\|_{L^p(\Omega)}\leq C\big(1+e^{C\|f\|_V^2}\big).
$$

\end{theorem}
\smallskip 
 
\section{On Stokes Operators}
\label{stokesappendix} 
\setcounter{equation}{0}

\noindent
We consider the homogeneous Stokes problem 
\begin{equation}
\label{STOKES}
\begin{cases}
-\Delta \uu+ \nabla p=\f, \quad &\text{in} \ \Omega,\\
\mathrm{div}\, \uu=0,\quad &\text{in} \ \Omega,\\
\uu=0, \quad &\text{on}\ \partial\Omega.
\end{cases}
\end{equation}
First, we introduce the Stokes operator as the map $\A:\V_\sigma\rightarrow\V_\sigma'$ such that
$$
\l \A \uu,\vv\r= (\nabla \uu,\nabla \vv ), 
\quad \forall \, \uu, \vv \in \V_\sigma,
$$
namely $\A$ is the canonical isomorphism from $\V_\sigma$ onto $\V_\sigma'$. We denote by $\A^{-1}: \V_\sigma' \rightarrow \V_\sigma$ the inverse map of the Stokes operator. That is,  given $\f \in \V_\sigma'$, there exists a unique $\uu=\A^{-1} \f \in \V_\sigma$ such that
$$
( \nabla \A^{-1} \f, \nabla \vv )=\l \f, \vv\r, \quad \forall \, \vv \in \V_\sigma.
$$
It follows that 
$
\| \f \|_{\sharp}:= \| \nabla \A^{-1} \f \|=\l \f,\A^{-1} \f \r^{\frac12}
$ 
is an equivalent norm on $\V_\sigma'$ and the chain rule
$$
\l \f_t (t), \A^{-1} \f (t)\r= \frac12 \ddt \| \f(t)\|_{\sharp}^2, \quad
\text{ a.e. }t\in (0,T),
$$ 
holds for any $\f\in H^{1}(0,T;\V_\sigma')$.
In order to recover the pressure $p$, the well-known De Rham result implies that, if $\f \in \H^{-1}(\Omega)$, there exists $p\in H$ (such that $\overline{p}=0$) such that $\nabla p= \Delta \uu+\f$ in the distributional sense. In addition, by \cite[Proposition 1.2]{Temam} we know that
\begin{equation}
\label{pH-1}
\| p\|\leq C \| \f\|_{\H^{-1}(\Omega)}.
\end{equation}
Let us now report the regularity theory of the Stokes problem \eqref{STOKES} (see \cite{CATT}). 
Assuming that $\f \in \H$, then there exist a unique $\uu \in \mathbf{H}^2(\Omega)\cap \V_\sigma$ and $p \in V$ (unique up to a constant) such that 
$-\Delta \uu+ \nabla p=\f$ almost everywhere in $\Omega$.
Moreover, there exists a constant $C$ such that
\begin{equation}
\label{stokesH2}
\| \uu\|_{\H^2(\Omega)}+ \| p\|_{V}\leq C\| \f\|.
\end{equation}
We denote by $P: \H\rightarrow \H_\sigma$ the Helmholtz-Leray orthogonal projection from $\H$ onto $\H_\sigma$. 
We recall that $P$ is a bounded operator from $\V$ into 
$\V\cap \H_\sigma$, namely there exists a positive constant $C$ such that
$$
\| P\vv\|_{\V}\leq C \|\vv\|_{\V}, \quad \forall \vv \in \V.
$$
We also report that $P\nabla v=\mathbf{0}$ for any $v\in V$.
Next, we consider the Stokes operator as an unbounded operator on $\H_\sigma$ with domain $D(\A)=\lbrace \uu\in \V_\sigma: \A \uu\in \H_\sigma\rbrace$. 
It is well known that $\A$ is a positive, unbounded, self-adjoint operator in $\H_\sigma$ with compact inverse (see, e.g., \cite{Temam}). In particular, we have
$$
\A \uu= P(-\Delta\uu ), \quad \forall \, \uu \in D(\A),\quad \text{where} 
\quad D(\A)=\mathbf{H}^2(\Omega)\cap \V_\sigma.
$$
Thanks to the above regularity results, we deduce that the operator $\A^{-1}: \H_\sigma \rightarrow \mathbf{H}^2(\Omega)\cap \V_\sigma$ is such that, for any $\f \in \H_\sigma$, there exist $\A^{-1} \f \in D(\A)$ and $p\in V$ which solve 
\begin{equation}
\label{stokesinverse}
-\Delta \A^{-1} \f+ \nabla p = \f.
\end{equation}
In turn, this entails that $\A \A^{-1}\f=\f$. Owing to \eqref{stokesH2}, we have 
\begin{equation}
\label{esthigh}
\|\A^{-1}\f\|_{\H^2(\Omega)}+\| p\|_V\leq C \|\f \|.
\end{equation}

We are now in position to find an $L^2$-estimate of the pressure $p$ in \eqref{stokesinverse} in terms of $\|\nabla \A^{-1} \f \|$.
Let us first report a preliminary interpolation result (see \cite{Necas}).

\begin{lemma}
\label{interpolationboundary}
Let $\Omega$ be a Lipschitz domain in $\mathbb{R}^{d}$, $d=2,3$, with compact boundary. Then, there exists a positive constant $C$ such that
\begin{equation}
\| f\|_{L^2(\partial \Omega)}\leq C \| f\|^{\frac12} \|f\|_V^{\frac12}, \quad \forall \, f\in V.
\end{equation}
\end{lemma}

\noindent
We have the following result.

\begin{lemma}
Let $d=2,3$ and $\f \in \H_\sigma$. Then, there exists a positive constant $C$ (independent of $\f$) such that
\begin{equation}
\label{pL2}
\| p\|\leq C \| \nabla \A^{-1}\f\|^\frac12 
\| \f\|^\frac12.
\end{equation}
\end{lemma}

\begin{proof}
Thanks to \eqref{pH-1}, we need to control $\|\f \|_{\H^{-1}(\Omega)}$ by means of $\| \f\|_\sharp$. To this end, let us consider $\vv \in \H_0^1(\Omega)$ with $\| \vv\|_{\H_0^1(\Omega)}\leq 1$. By exploiting the integration by parts, we find
\begin{align*}
(\f,\vv)&= (P(-\Delta) \A^{-1}\f, \vv)\\
&=(-\Delta \A^{-1} \f, P\vv)\\
&= (\nabla \A^{-1}\f, \nabla P\vv)- 
\int_{\partial \Omega} \nabla \A^{-1}\f \n \cdot P\vv \, \d \sigma.
\end{align*}
We recall that the classical trace theorem implies 
$\| P\vv\|_{L^2(\partial \Omega)}\leq C\| P\vv\|_V$. In addition, by the properties of the Helmholtz-Leray operator and the Poincar\'{e} inequality, we have $\| P\vv\|_{\V}\leq C\| \vv\|_{\H_0^1(\Omega)}$. Then, we deduce that
$$
\| \f\|_{\H^{-1}(\Omega)}\leq C \| \nabla \A^{-1}\f\|+ 
C \| \nabla \A^{-1}\f\|_{\mathbf{L}^2(\partial \Omega)}.
$$
An application of Lemma \ref{interpolationboundary}, together with \eqref{esthigh}, implies that
$$
\| \f\|_{\H^{-1}(\Omega)}\leq C \| \nabla \A^{-1}\f\|+ 
C \| \nabla \A^{-1}\f\|^{\frac12} \| \f\|^{\frac12}.
$$
Thus, the desired inequality \eqref{pL2} immediately follows.
\end{proof}

\smallskip
 
Finally, we consider the homogeneous Stokes problem with nonconstant viscosity depending on a given measurable function $\vp$. The system reads as follows
\begin{equation}
\label{STOKESVISC}
\begin{cases}
-\mathrm{div} (\nu(\varphi)D \uu)+ \nabla \pi=\f, \quad &\text{in} \ \Omega,\\
\mathrm{div}\, \uu=0,\quad &\text{in} \ \Omega,\\
\uu=0, \quad &\text{on}\ \partial\Omega,
\end{cases}
\end{equation}
where the coefficient $\nu$ fulfils the assumptions stated in Section \ref{2}.
We report a regularity result whose proof has been provided in \cite[Sec. 4, Lemma 4]{A}.

\begin{theorem}
\label{STOKESVISCREG}
Let $d=2$, $\vp \in W^{1,r}(\Omega)$, with $2<r\leq \infty$, and 
$\f\in \mathbf{L}^p(\Omega)$, with $1\leq p<\infty$.
Assume that $\uu \in \mathbf{V}_\sigma$ is a weak solution to \eqref{STOKESVISC}, i.e.
$$
(\nu(\varphi)D \uu,D\vv)=( \f,\vv ),\quad \forall \, \vv \in \mathbf{V}_\sigma.
$$
Then, there exists $C>0$, depending on $r$ and $p$, such that
\begin{equation}
\label{regstokes}
\| \uu\|_{\mathbf{W}^{2,p'}(\Omega)} \leq
C
\big( 1+\| \nabla \vp\|_{\mathbf{L}^r(\Omega)} \big)
\big( \| \f\|_{\mathbf{L}^p(\Omega)}+\|\nabla \uu \|\big),
\end{equation}
where $\frac{1}{p'}=\frac{1}{p}+\frac{1}{r}$, provided that $p'>1$.
\end{theorem}
 
\section{A Product Estimate in Two Dimensions}
\label{Log-est}
\setcounter{equation}{0}

\noindent
We report here a logarithmic estimate of the product of two functions in two dimensions. The following proof is based on an idea developed in  \cite{DFJ2005} and \cite{Titi1986} to control the convective term of the Navier-Stokes equations.

\begin{proposition}
\label{result1}
Let $\Omega$ be a bounded domain in $\mathbb{R}^2$ with smooth boundary. Assume that $f\in V$ and $g\in V$. Then, there exists a positive constant $C$ such that
\begin{equation}
\label{logprod1}
\| f g\|\leq C \| f\|_{V}
\| g\|  \Big[ \log \Big( \mathrm{e} \frac{\| g\|_{V}}{\| g\|} \Big) \Big]^\frac12.
\end{equation}
\end{proposition}

\begin{proof}
Let us consider the operator $A_1=-\Delta + I$ on $H$ with domain $D(A_1)=\lbrace u\in H^2(\Omega): \partial_{\mathbf{n}}u=0$ on $\partial \Omega\rbrace$ defined in Appendix \ref{Neumann-Laplace}. By the spectral theory, there exists a sequence of positive eigenvalues $\lambda_k$ ($k\in \mathbb{N}$) associated with $A_1$ such that $\lambda_1=1$, $\lambda_{k}\leq \lambda_{k+1}$ and $\lambda_{k}\rightarrow \infty$ as $k$ goes to $\infty$. The sequence of eigenfunctions $w_k\in D(A_1)$ such that $A_1 w_k=\lambda_k w_k$ forms an orthonormal basis in $L^2(\Omega)$ and an orthogonal basis in $H^1(\Omega)$. In particular, we have the representation
$$
f= \sum_{k=1}^\infty (f,w_k) w_k.
$$
Let us fix $N \in \mathbb{N}$ whose value will be chosen later. We write $f$ as follows
\begin{equation}
\label{decompsition}
f=\sum_{n=0}^N f_n + f_N^{\bot},
\end{equation}
where
$$
f_n=\sum_{k: \mathrm{e}^n\leq \sqrt{\lambda_k}<\mathrm{e}^{n+1}} 
(f,w_k) w_k, \quad f_{N}^{\bot}= \sum_{k:\sqrt{\lambda_k} \geq \mathrm{e}^{N+1}} (f,w_k) w_k.
$$
By using the above decomposition, the H\"{o}lder inequality and subsequently \eqref{LADY} and \eqref{Agmon2d},  we find
\begin{align}
\| f g\| &\leq 
\sum_{n=0}^N \| f_n g \|+
\| f_{N}^{\bot} g \| \notag \\
& \leq \sum_{n=0}^N \| f_n\|_{L^\infty(\Omega)} \| g\| 
+ \| f_N^{\bot}\|_{L^4(\Omega)}
\| g\|_{L^4(\Omega)} \notag \\
&\leq C \sum_{n=0}^N \| f_n\|^\frac12
\| f_n\|_{H^2(\Omega)}^\frac12 \| g\|+ C \| f_N^{\bot}\|^\frac12
\| f_N^{\bot}\|_{V}^\frac12
\| g\|^\frac12 \| g\|_{V}^\frac12.\label{est1}
\end{align}
We now observe that
$$
\| f_n\|_{L^2(\Omega)}^2= \sum_{k: \mathrm{e}^n\leq \sqrt{\lambda_k}<\mathrm{e}^{n+1}} |(f,w_k)|^2 \leq \frac{1}{\mathrm{e}^{2n}} \sum_{k: \mathrm{e}^n\leq \sqrt{\lambda_k}<\mathrm{e}^{n+1}} \lambda_k |(f,w_k)|^2 = \frac{1}{\mathrm{e}^{2n}} \| f_n\|_{H^1(\Omega)}^2.
$$
Here we have used the fact that $D(A_1^\frac12)=V$. 
Observing that $f_n$ is a finite sum of $w_k$'s, by the regularity theory of the Neumann problem, we have
\begin{align*}
\| f_n\|_{H^2(\Omega)}^2 &\leq C \| A_1 f_n\|_{L^2(\Omega)}^2= C \sum_{k: \mathrm{e}^n\leq \sqrt{\lambda_k}<\mathrm{e}^{n+1}} \lambda_k^2 |(f,w_k)|^2\\
&\leq C  \sum_{k: \mathrm{e}^n\leq \sqrt{\lambda_k}<\mathrm{e}^{n+1}} \mathrm{e}^{2(n+1)} \lambda_k |(f,w_k)|^2\\
&\leq C \mathrm{e}^{2(n+1)} \| f_n\|_{V}^2.
\end{align*}
Then, we deduce that 
$$
\| f_n\|^\frac12
\| f_n\|_{H^2(\Omega)}^\frac12 \leq C \mathrm{e}^\frac12 \| f_n\|_{V}.
$$
On the other hand, reasoning as above, we have
$$
\| f_{N}^{\bot}\|_{L^2(\Omega)}^2\leq \frac{1}{\mathrm{e}^{2(N+1)}} \| f_{N}^{\bot}\|_{V}^2.
$$
Combining the above inequalities in \eqref{est1}, and applying the Cauchy-Schwarz inequality, we get
\begin{align}
\| fg\| &\leq 
C \sum_{n=0}^N \mathrm{e}^\frac12 \| f_n\|_{V} \| g\| + 
C \frac{1}{\mathrm{e}^{\frac{N+1}{2}}} \| f_N^{\bot}\|_{V} \| g\|^\frac12 \| g\|_{V}^\frac12 \notag\\
&\leq C  \| g\|
\Bigg(  \sum_{n=0}^N \mathrm{e}^\frac12 \| f_n\|_V + \frac{1}{\mathrm{e}^{\frac{N+1}{2}}}
\frac{\| g\|_V^\frac12}{\| g\|^\frac12} \| f_N^{\bot}\|_V\Bigg) \notag\\
&\leq C \| g\|
\Bigg(   \mathrm{e} (N+1) + \frac{1}{\mathrm{e}^{N+1}}
\frac{\| g\|_{V}}{\| g\| } \Bigg)^\frac12  \Bigg( \sum_{n=0}^N 
\| f_n\|_{V}^2  +\| f_N^{\bot}\|^2_V
\Bigg)^\frac12 \notag\\
&\leq  C \| g\|
\Bigg(   \mathrm{e} (N+1) + \frac{1}{\mathrm{e}^{N+1}}
\frac{\| g\|_V}{\| g\|} \Bigg)^\frac12 \| f\|_{V}. \label{est2}
\end{align}
Now, we choose the integer $N$ so that
$$
\ln \Big( \mathrm{e} \frac{\| g\|_V}{\| g\|}\Big) \leq N+1 < 1+ \ln \Big(\mathrm{e} \frac{\| g\|_V}{\| g\|}\Big)
$$
By using the above choice of $N$ in \eqref{est2}, we eventually infer that
$$
\| f g\| \leq 
C\| g\| \| f\|_V
\Big[  \mathrm{e} \ln \Big( \mathrm{e}^2 \frac{\| g\|_{V}}{\| g\|}\Big) +\frac{1}{\mathrm{e}} \Big]^\frac12 
,
$$
which implies the desired conclusion.
\end{proof}

\noindent
For the purpose of this work we state an immediate generalization of \eqref{logprod1}, whose proof which can be inferred from that of Proposition \eqref{result1} is left to the interested reader. 

\begin{proposition}
\label{result2}
Let $\Omega$ be a bounded domain in $\mathbb{R}^2$ with smooth boundary. Assume that $f\in V$, $\g\in \V$ and $h\in V$. Then, there exists a positive constant $C$ such that
\begin{equation}
\label{logprod2}
\| f \g\| \leq C \| f\|_{V}
\Big(\| \g\| +\| h\| \Big) \Big[ \log \Big( \mathrm{e} \frac{\| \g\|_{\V}+\|h\|_V}{\| \g\|+\| h\|} \Big) \Big]^\frac12.
\end{equation}
\end{proposition}
 
\bigskip
\noindent 
\textbf{Acknowledgments.} 
The authors wish to thank Yining Cao for her careful reading of the manuscript and her helpful remarks. 
Part of this work was done while the first author was visiting the Laboratoire de Math\'{e}matiques et Applications at the Universit\'{e} de Poitiers whose hospitality is gratefully acknowledged.
The research experience was supported by 
LIA-LYSM: AMU-CNRS.ECM-INdAM funding. 
The work of A. Giorgini was also supported by GNAMPA-INdAM Project
 ``Analisi matematica di modelli a interfaccia diffusa per fluidi complessi" and by the project Fondazione Cariplo-Regione Lombardia MEGAsTAR ``Matematica d'Eccellenza in biologia ed ingegneria come acceleratore di una nuova strateGia per l'ATtRattivit\`{a} dell'ateneo pavese".
The work of R. Temam was partially supported by the National Science Foundation [grant number NSF-DMS-1510249] and by the Research Fund of Indiana University.

\end{document}